\documentclass{ams-ut}

\usepackage[T1]{fontenc}
\usepackage{lmodern}

\title[{Cellularity of endomorphism algebras of tilting objects}]{Cellularity of endomorphism algebras\protect\\of tilting objects}

\author[G. Bellamy]{Gwyn Bellamy}
\address{School of Mathematics and Statistics, University Place, University of Glasgow, Glasgow, G12 8QQ, UK}
\email{gwyn.bellamy@glasgow.ac.uk}

\author[U. Thiel]{Ulrich Thiel}
\address{Department of Mathematics, University of Kaiserslautern, Postfach 3049, 67653 Kaiserslautern, Germany}
\email{thiel@mathematik.uni-kl.de}

\date{Feb 29, 2022}

\newcommand{\bc}{\mathbf{c}}
\newcommand{\bbD}{\mathbb{D}}
\newcommand{\bbN}{\mathbb{N}}

\DeclareMathOperator{\Sub}{Sub}
\newcommand{\bdm}{\begin{displaymath}}
\newcommand{\edm}{\end{displaymath}}
\newcommand{\bthm}{\begin{thm}}
\newcommand{\ethm}{\end{thm}}
\newcommand{\blem}{\begin{lem}}
\newcommand{\elem}{\end{lem}}
\newcommand{\bcor}{\begin{cor}}
\newcommand{\ecor}{\end{cor}}
\newcommand{\beq}{\begin{equation}\label}
\newcommand{\eeq}{\end{equation}}
\newcommand{\bprop}{\begin{prop}}
\newcommand{\eprop}{\end{prop}}
\newcommand{\bdefn}{\begin{defn}}
\newcommand{\edefn}{\end{defn}}

\newcommand{\h}{\mathfrak{h}}

\let\mc\mathcal

\DeclareMathOperator{\Ind}{Ind}

\renewcommand{\mod}{\textrm{mod}}

\newcommand{\op}{\mathrm{op}}

\newcommand{\ds}{\dots}

\begin{document}

\begin{abstract}
We show that, in a highest weight category with duality, the endomorphism algebra of a tilting object is naturally a cellular algebra. Our proof generalizes a recent construction of Andersen, Stroppel, and Tubbenhauer~\cite{AST}. This result raises the question of whether all cellular algebras can be realized in this way. The construction also works without the presence of a duality and yields standard bases, in the sense of Du and Rui, which have similar combinatorial features to cellular bases. As an application, we obtain standard bases—and thus a general theory of ``cell modules''—for  Hecke algebras associated to finite complex reflection groups (as introduced by Broué, Malle, and Rouquier) via category~$\mathcal{O}$ of the rational Cherednik algebra. For real reflection groups these bases are cellular.
\end{abstract}

\maketitle

\section{Introduction}

Cellular algebras were introduced by Graham and Lehrer \cite{Graham-Lehrer-Cellular} in 1996 and have since then become a major topic in representation theory. Simply put, these are algebras admitting a basis and an anti-involution satisfying some specific combinatorial properties. The prime example, on which these properties were modeled, are Hecke algebras of type $A$. The structure of a cellular algebra provides, via the associated cell modules, an effective way to tackle the representation theory of the algebra. This often leads to (or explains) interesting combinatorics.

The list of examples of algebras known to admit a cellular structure has grown considerably in recent years. For instance, many algebras admitting a diagrammatic interpretation—such as Temperley–Lieb and Brauer algebras—are naturally cellular with the basis given by a collection of diagrams and the anti-involution flipping the diagrams.

However, many algebras for which there is no obvious diagrammatic interpretation also admit a cellular structure. This raises the question of whether there is a general underlying property of all these algebras that implies their cellularity.

A possible answer lies in recent work of Andersen, Stroppel, and Tubbenhauer~\cite{AST} (a precursor of this paper is \cite{ALZ}). They observe that many examples of cellular algebras arise as endomorphism algebras of tilting modules for quantum groups. In this setting, they developed a construction of cellular bases for endomorphism algebras of arbitrary tilting modules, which recovers the known cases.
This approach has several appealing features: it is general, it is explicit (to some extent), and it involves choices which actually lead to a whole family of cellular bases, so that even in classical examples one obtains new cellular bases.

As already mentioned in their paper, it is clear that their construction works in greater generality than for tilting modules of quantum groups—it should essentially work for tilting objects in any highest weight category. However, certain technical difficulties arise from trying to work in complete generality.

The key limitation of the construction of cellular bases in \cite{AST}, as noted in \cite[Section 5A.7]{AST}, is the reliance on weight space decompositions of modules for quantum groups in order to prove the crucial result \cite[Theorem 3.1]{AST}. Specifically, by restricting morphisms to weight spaces, one is able to define a filtration on Hom-spaces between modules, which plays an important role in the construction. However, weight spaces do not exist in arbitrary highest weight categories. Another necessary condition is that the (isomorphism classes of) indecomposable tilting objects are in (canonical) bijection with the simple objects, as in Ringel's theory \cite{Ringel-Filtrations}. Again, this does not hold in general highest weight categories.

The purpose of this paper is to specify a general categorical setup to which the construction from \cite{AST} can be generalized. To this end, we introduce in \Autoref{section_tilting} the notion of a ``standard category'' which is simply a category with standard, costandard, and tilting objects behaving in the desired way. Our assumptions are rather mild and include (split) highest weight categories \cite{CPS} with finitely many simple objects and the recently introduced lower finite highest weight categories of Brundan and Stroppel \cite{BrundanStroppel}. To overcome the key limitation caused by the reliance on weight spaces, we define instead a categorical filtration on Hom-spaces, see \eqref{cellular_newsection_key_modification}. This new filtration enables one to repeat the proofs of \cite{AST} verbatim, see \Autoref{cellularity_section}.

The involution, which is part of the cellular structure, comes in our general setting from a duality on the category. Not all dualities give rise to an involution and this issue leads us to the notion of a ``standard duality''. First, a standard duality should exchange standard and costandard objects. This implies that all tilting objects are self-dual and the duality induces an isomorphism on the endomorphism algebra of a tilting object. But, in general, the square of this isomorphism is just \emph{some} inner automorphism. To ensure that we get an involution, the duality needs to ``fix'' the indecomposable tilting objects in a categorical sense. We will discuss this in \Autoref{subsec_duality}.

Finally, we note that even in the absence of a duality on the category one can still perform the construction to obtain combinatorial bases that allows one to define cell modules which will encapsulate the same combinatorics as for cellular bases. Such bases are called \word{standard bases} and have been introduced by Du and Rui \cite{DuRui} in 1998.\footnote{We note that being cellular (i.e. admitting a cellular basis) is a restrictive property for an algebra whereas having a standard basis is not. In other words, every algebra admits at least one standard basis; see~\cite{KonigXiCellular}. Nonetheless, a particular choice of standard basis may still encode interesting combinatorics.} \\

Our main result is thus the following generalization of the main result of \cite{AST} (the precise versions of the two statements are \Autoref{end_tilting_is_cellular,cellularity_thm_precise}):

\begin{thm*}
Let $\mathcal{C}$ be a standard category.
\begin{enumerate}
  \item \label{main_thm_part1} For any tilting object $T \in \mathcal{C}$ one can construct a standard basis on the algebra $\End_{\mathcal{C}}(T)$ by ``factorization through the indecomposable tilting objects''.
  \item If $\mathcal{C}$ is equipped with a standard duality $\bbD$, then the construction given in (\ref{main_thm_part1}) can be done in such a way that the resulting basis is cellular with respect to the anti-involution on $\End_{\mathcal{C}}(T)$ induced by $\bbD$. In particular, $\End_{\mathcal{C}}(T)$ is a cellular algebra.
\end{enumerate}
\end{thm*}

It follows from the theorem and the construction of the standard bases that the full subcategory $\mathcal{C}^t$ of tilting objects in $\mathcal{C}$ is a (strictly) object-adapted cellular category in the sense of Elias--Lauda \cite{EliasLauda}, see also \cite[\S11]{EMTW} and \cite[Remark~5.7]{BrundanStroppel}.\footnote{This has been noticed recently by Andersen \cite{Andersen-Cellular} in a more special setting.} In \cite[Theorem~5.10]{BrundanStroppel}, the authors prove a similar theorem, in the context of lower finite highest weight categories with Chevalley duality.\\

Motivated by this result we would like to pose the following fundamental question:% about the nature of cellular algebras:

\begin{quest*}
Is every finite-dimensional split cellular algebra the endomorphism algebra of a tilting object?
\end{quest*}

%\footnote{We assume for this question that the algebra is finite-dimensional over a field $K$ and split over $K$, meaning every simple module remains simple under field extensions.}

So far, the concept of cellularity has defied all attempts at a categorical characterization. We hope this question may lead to new insights. Currently, it is not clear how to give a definitive answer. Moreover, if one can realize a cellular algebra as the endomorphism algebra of a tilting object, it is not clear whether the above construction can be used to construct \emph{all} cellular structures on the algebra.

An important test case for our question is the partition algebra \cite{MartinSaleur, Martin-PartitionAlgebra} which was proven to be cellular in \cite{Xi-PartitionAlgebra-Cellular}. As far as we are aware, it is currently not known if the partition algebra can be realized as the endomorphism algebra of a tilting object. \\

Aside from raising this question we have specific applications in mind requiring our generalization. The first application, which we address in \Autoref{hecke_example}, concerns Hecke algebras. Let $\mathcal{H}_{\mathbf{q}}(W)$ be the Hecke algebra associated to a finite complex reflection group $W$ at \textit{arbitrary parameter} $\mathbf{q}$, as introduced by Broué, Malle, and Rouquier \cite{Malle-Rouquier-2003}. This algebra is a generalization of the ``classical'' Hecke algebra \cite{GeckPfeiffer} associated to a finite Coxeter group. It is known from \cite{GGOR} that $\mathcal{H}_{\mathbf{q}}(W)$ arises as the endomorphism algebra of a tilting object in the category $\mathcal{O}$ of the rational Cherednik algebra $H_{1,\mathbf{c}}(W)$. The latter algebra was introduced by Etingof and Ginzburg \cite{EG}, with $\mathbf{c}$ a logarithm of $\mathbf{q}$. Our theorem implies that $\mathcal{H}_{\mathbf{q}}(W)$ carries a natural standard basis. Moreover, if $W$ is a real reflection group (i.e. a finite Coxeter group), then $\mathcal{O}$ is equipped with a ``standard duality''  making this standard basis a cellular basis. That is, $\mathcal{H}_{\mathbf{q}}(W)$ is cellular. The cellularity of Hecke algebras associated to finite Coxeter groups is a deep result that was proven by Geck \cite{Geck:2007ki} in 2007. Our approach to cellularity is completely different to the one Geck has taken and we make the following observations:
\begin{enumerate}
\item Geck must assume Lusztig's conjectures P1--P15 (which are still open for unequal parameters). Our result covers the unequal parameter case without further assumptions.

\item Geck covers more general base rings—especially of positive characteristic—whereas our result only applies to Hecke algebras over the complex numbers. It would be interesting to know whether there is a modular analogue of category $\mathcal{O}$ of the rational Cherednik algebra that can be used to generalize our result to positive characteristic. The diagrammatic Cherednik algebra of Webster \cite{WebsterDiagrammatic} and Bowman \cite{BowmanBasis} is a likely candidate for this.

\item Currently, we do not have an explicit description of our bases and do not know how they relate to Geck's (and other) cellular bases.

\item Our construction works for the Hecke algebra associated to an arbitrary complex reflection group $W$, not just for finite Coxeter groups. By Du--Rui \cite{DuRui}, we obtain a general theory of ``cell modules'' for these Hecke algebras. We think this is interesting in light of the ``spetses'' program \cite{Spetses1, Spetses2, BR}.

\item Assuming Lusztig's conjectures P1--P15, work of Chlouveraki--Gordon--Griffeth \cite{GordonGriffChlo} implies that the ``costandard modules'' resulting from our construction coincide with Geck's cell modules for Hecke algebras associated to Coxeter groups\footnote{We would like to thank Chris Bowman--Scargill for bringing this to our attention.}; see \Autoref{rem:cellHeckeGeck}.

\item In \cite{BowmanBasis}, Bowman shows for the infinite series $G(m,1,n)$ that the Hecke algebra $\mathcal{H}_{\mathbf{q}}(W)$ has (infinitely many) graded cellular structures. This is generalized to the groups $G(m,p,n)$ in \cite{HuMathasRostam}. We do not know how these graded bases relate to our bases.

\end{enumerate}

Another application, that we discuss in \cite{corepaper}, is the following. In \cite{hwtpaper} we studied finite-dimensional graded algebras $A$ admitting a triangular decomposition. Typical examples are restricted enveloping algebras, Lusztig’s small quantum groups, hyperalgebras, finite quantum groups, and restricted rational Cherednik algebras, see \cite{hwtpaper}. As we argue in \cite{corepaper}, these algebras are generally not cellular. However, using our main theorem we can obtain a natural cellular algebra associated to $A$. Namely, under the assumption that $A$ is self-injective—which is true in all the above mentioned examples—it follows from \cite{hwtpaper} that the category $\mathcal{G}(A)$ of graded finite-dimensional $A$-modules is a standard category (with infinitely many simple objects). Moreover, we have shown that $A \in \mathcal{G}(A)$ is a tilting object. Noting that $\End_{\mathcal{G}(A)}(A)^{\mathrm{op}}$ is isomorphic to the subalgebra $A_0$ of degree-0 elements of $A$, we obtain from our main theorem a standard basis of $A_0$. When $A$ is graded symmetric and there is a ``triangular'' anti-involution on $A$, we show in \cite{corepaper} that $\mathcal{G}(A)$ is equipped with a standard duality making $A_0$ cellular. The degree-0 subalgebra $A_0$ is interesting as it remembers the graded decomposition numbers of $A$; see \cite{corepaper}.

\begin{rem*}
A first draft of our results appeared as part of the preprint \cite{corepaperold}. However, we realized that a naive duality does not usually give rise to the required anti-involution, and it was not clear what the correct notion of duality should be. In the meantime, Brundan and Stroppel \cite{BrundanStroppel} and Andersen \cite{Andersen-Cellular} also developed  generalizations of the construction of cellular bases from \cite{AST}. We think that our categorical filtration on the Hom-spaces, as well as the clarification of the precise assumptions needed for \Autoref{end_tilting_is_cellular,cellularity_thm_precise} to hold, is a worthwhile complement to their work.
\end{rem*}

\subsection*{Acknowledgments}
We would like to thank J. Brundan, C. Bowman--Scargill, G.~Malle, J. Schmitt, C. Stroppel, and E. Thorn for comments on a preliminary version of this paper. We would especially like to thank S. Koenig for patiently answering several questions. We would also like to thank the referee of \textit{Adv. Math.} for valuable comments. The first author was partially supported by EPSRC grant EP/N005058/1. The second author was partially supported by the DFG SPP 1489, by a Research Support Fund from the Edinburgh Mathematical Society, and by the Australian Research Council Discovery Projects grant no. DP160103897. This work is a contribution to the SFB-TRR 195 `Symbolic Tools in Mathematics and their Application' of the German Research Foundation (DFG).

\tableofcontents

\section{Standard categories} \label{section_tilting}

In this section we will specify the categorical setup to which we can generalize the construction of the bases from \cite{AST}. We require categories with standard, costandard, and tilting objects behaving in the desired way---as in (sufficiently nice) highest weight categories \cite{CPS}. However, we will not assume that our categories have enough injectives—and thus we will refrain from the usual assumption made in highest weight categories that there exists a filtration by costandard objects on indecomposable injective envelopes—to be able to include examples such as categories of finite-dimensional modules for quantum groups in positive characteristic, as considered in~\cite{AST} (see also \Autoref{uq_mod_examples}). We will instead impose specific Ext-vanishing conditions which imply the desired features. For the sake of this paper we will call such categories ``standard categories'' because they deal with standard objects and eventually lead to ``standard bases''. Split highest weight categories \cite{CPS} with finitely many simple objects and lower finite highest weight categories in the sense of Brundan and Stroppel \cite{BrundanStroppel} fit into our framework, see \Autoref{finite_hwc_has_everything} and \Autoref{lfhw_has_everything}.

\subsection{Locally finite abelian categories}

We begin with standard categorical results that will be required in the next section.

\begin{ass} \label{ass_lf}
We assume that $\mathcal{C}$ is an essentially small and locally finite abelian category over a field $K$.
\end{ass}

By \word{locally finite} we mean, as in \cite[\S1.8]{EGNO}, that all objects are of finite length (i.e. admit a composition series) and that $\Hom_K(X
,Y)$ is a finite-dimensional $K$-vector space for all $X,Y \in \mathcal{C}$. By the Jordan–Hölder theorem (see \cite[\S4.5, Theorem 1]{Pareigis}) the multiplicity of a simple object $S$ as a quotient in a composition series of $X$ is independent of the choice of the composition series. We denote this multiplicity by $\lbrack X \colon S \rbrack$. Since~$\mathcal{C}$ is essentially small, the class $\Sub(X)$ of subobjects of $X$ (equivalence classes of monomorphisms into $X$) is a \emph{set}. The usual partial order $\leq$ on subobjects makes $\Sub(X)$ into a partially ordered set. Since $X$ is of finite length, this partially ordered set is artinian and noetherian, see \cite[\S4.5, Corollary 1]{Pareigis}. In particular, the category~$\mathcal{C}$ is artinian and noetherian in the sense of Gabriel \cite[\S II.4]{Gabriel}.

The supremum in $\Sub(X)$ of a (small) family $(U_i)_{i \in I}$ of subobjects $U_i$ of $X$ is denoted by $\cup_{i \in I} U_i$ if it exists and is called the \word{union} of the family. Similarly, the infimum is denoted by $\cap_{i \in I} U_i$ and is called the \word{intersection}. Since $\mathcal{C}$ is abelian, unions and intersections of \emph{finite} families exist, see \cite[\S4.3, Lemma 5]{Pareigis}. In particular, $\Sub(X)$ is a \emph{lattice} with respect to union and intersection. But since $\Sub(X)$ is noetherian and artinian, in fact \word{arbitrary} (i.e. not necessarily finite) unions and intersections exist, i.e. $\Sub(X)$ is a \word{complete} lattice, as shown below.

\begin{lem}\label{lem:comletelattice}
A noetherian lattice with minimal element (or an artinian lattice with maximal element) is complete.
\end{lem}

\begin{proof}
We could not find a reference for this (likely well-known) fact so give a proof. Let $(L,\leq)$ be a noetherian lattice with minimal element $0$. Let $X \subs L$ be a subset. We show that $X$ has a supremum. If $X$ is empty, then $\bigcup X = 0$ is the minimal element. So, assume $X$ is not empty. Since $L$ is a lattice, the supremum of every finite set exists and therefore we can define
\[
Y \coloneqq \{ \bigcup F \mid F \subs X \tn{ is finite and non-empty} \} \;.
\]
Since $X$ is non-empty, $Y$ is also non-empty and since $L$ is noetherian, $Y$ has a maximal element $m$. Since $m \in Y$, we have $m = \bigcup F$ for some finite non-empty subset $F \subs X$. We claim that $m$ is an upper bound of $X$, i.e., $x \leq m$ for any $x \in X$. Since $F \cup \{ x \}$ is a finite non-empty subset of $X$, we have $\bigcup \left( F \cup \{ x \} \right) \in Y$. Moreover, $m = \bigcup F \leq \bigcup \left( F \cup \{ x \} \right) \leq m$ since $m$ is maximal in $Y$. Hence, $x \leq m$. This proves that $m$ is an upper bound of $X$. Let $u$ be any other upper bound of $X$. Then $u$ is an upper bound of $F \subs X$, hence $m = \bigcup F \leq u$. Hence, $m$ is the least upper bound of $X$, i.e., $m = \bigcup X$. This proves that the supremum of any subset of $L$ exists. It then follows automatically that any non-empty subset $X \subs L$ has an infimum as well since this is given by
\[
\bigcap X = \bigcup \{ y \in L \mid y \leq x \ \forall x \in X \} \;.
\]
Dually, one can prove that an artinian lattice with maximal element is complete.
\end{proof}

In the context of highest weight categories one usually assumes the \emph{Grothendieck condition}; see \cite{CPS}. We want to argue that this holds for any noetherian category $\mathcal{C}$, in particular for the categories we consider here. First, recall that an object $X \in \mathcal{C}$ is said to satisfy the Grothendieck condition if the union of any directed family $(U_i)_{i \in I}$ of subobjects of $X$ exists and for any further subobject $V$ of $X$ the relation
\[
V \cap \left( \bigcup_{i \in I} U_i \right) = \bigcup_{i \in I} (V \cap U_i)
\]
holds.\footnote{In the formulation of the Grothendieck condition by Cline--Parshall--Scott in \cite{CPS} it is not assumed that the families $(U_i)_{i \in I}$ are \emph{directed}. This seems to be a minor inaccuracy since in \cite{CPS} it is only assumed that the union of a directed set of subobjects exists, so a general union is not assumed to exist. The definition of the Grothendieck condition given here is the standard one, see e.g. \cite[\S4.7]{Pareigis}.} The category $\mathcal{C}$ is said to satisfy the Grothendieck condition if all its objects satisfy it.

\begin{lem}
A noetherian abelian category satisfies the Grothendieck condition.
\end{lem}

\begin{proof}
Let $\mathcal{C}$ be a noetherian abelian category.
 It follows from \cite[Théorème II.4.1]{Gabriel} that $\mathcal{C}$ is equivalent to the full subcategory of noetherian objects of some locally noetherian category $\wh{\mathcal{C}}$. Here, \emph{locally noetherian} means that the category has exact direct limits and has a small family of generators which are noetherian. Let $X \in \mathcal{C}$ and let $(U_i)_{i \in I}$ be a directed family of subobjects of $X$. In $\wh{\mathcal{C}}$ we can form the union $\bigcup_{i \in I} U_i$. Since this is a subobject of the noetherian object $X$, it is again noetherian and belongs to $\mathcal{C}$. Because this object is noetherian and the family is directed, there must be some index $j$ such that $\bigcup_{i \in I} U_i = U_j$. For any further subobject $V$ of $X$ we then have
\[
V \cap \left( \bigcup_{i \in I} U_i \right) = V \cap U_j = \bigcup_{i \in I} (V \cap U_i) \;.
\]
The second equality above follows from the fact that since $\bigcup_{i \in I} U_i = U_j$ we have $U_i \leq U_j$ for all $i \in I$, hence $V \cap U_i \leq V \cap U_j$ for all $i \in I$ which implies $\bigcup_{i \in I} (V \cap U_i) \leq V \cap U_j$, and the other inclusion is clear. Hence $\mathcal{C}$ satisfies the Grothendieck condition.
\end{proof}

Since $\mathrm{Sub}(X)$ is complete, we can define the \word{radical} and \word{socle} of $X$ as
\begin{equation}
\Rad(X) \coloneqq \bigcap_{ \mathclap{\substack{M \in \Sub(X) \\ M \text{ is maximal}}}} M \quad \quad \text{and} \quad
\Soc(X) \coloneqq \bigcup_{ \mathclap{\substack{N \in \Sub(X) \\ N \tn{ is minimal}}}} N \;,
\end{equation}
respectively. The quotient $\Hd(X) \coloneqq X / \Rad(X)$ is called the \word{head} of $X$.

\subsection{Simple objects and (co)standard objects}

\begin{ass} \label{ass_simples}
There is a complete set $\{ L(\lambda) \}_{\lambda \in \Lambda}$ of representatives of isomorphism classes of simple objects of $\mathcal{C}$ indexed by a set $\Lambda$ equipped with a partial order~$\leq$.
\end{ass}

Note that we do not assume that $\Lambda$ is finite. The partial order is used in the following definitions.

\begin{defn}\label{defn:highest_weight}
We say that an object $M \in \mathcal{C}$ is of \word{highest weight} $\lambda \in \Lambda$ if $\lbrack M \colon L(\mu) \rbrack \neq 0$ implies $\mu \leq \lambda$ and $\lbrack M \colon L(\lambda) \rbrack = 1$.
\end{defn}

\begin{defn}\label{defn:(co)standard}
A \word{costandard object} for $L(\lambda)$ is an object $\nabla(\lambda)$ such that $\Soc \nabla(\lambda) \simeq L(\lambda)$ and all composition factors $L(\mu)$ of $\nabla(\lambda)/\Soc \nabla(\lambda)$ satisfy $\mu < \lambda$. Dually, a \word{standard object} is an object $\Delta(\lambda)$ such that $\Hd \Delta(\lambda) \simeq L(\lambda)$ and all composition factors $L(\mu)$ of $\Rad \Delta(\lambda)$ satisfy $\mu < \lambda$.
\end{defn}

Both $\Delta(\lambda)$ and $\nabla(\lambda)$ are of highest weight $\lambda$. Note that since $\Hd \Delta(\lambda)$ is simple, it follows that $\Rad \Delta(\lambda)$ is the unique maximal subobject of $\Delta(\lambda)$. This implies that the epimorphism $\pi \colon \Delta(\lambda) \twoheadrightarrow \Hd \Delta(\lambda)$ is \word{essential}, i.e. if $\varphi \colon X \to \Delta(\lambda)$ is any morphism such that $\pi \circ \varphi$ is an epimorphism, then $\varphi$ must be an epimorphism. Dually, $\Soc \nabla(\lambda)$ is the unique minimal subobject and the monomorphism $\Soc \nabla(\lambda) \hookrightarrow \nabla(\lambda)$ is \emph{essential}.

\begin{rem}
Obviously, $L(\lambda)$ itself is a (co)standard object. \Autoref{ass_ext_vanishing} below will impose more conditions that usually rule out this possibility. More interesting (co)standard objects arise as follows. Suppose that $L(\lambda)$ admits an injective hull $I(\lambda)$ in $\mathcal{C}$. Then there is a unique largest subobject $\nabla(\lambda)$ of $I(\lambda)$ which is a costandard object for $L(\lambda)$. Namely, let $\mathcal{U}$ be the set of non-zero subobjects $U$ of $I(\lambda)$ with the property that all composition factors $L(\mu)$ of $U/L(\lambda)$ satisfy $\mu < \lambda$ (note that $L(\lambda)$ is the unique minimal subobject of $I(\lambda)$). By definition, all objects in $\mathcal{U}$ are costandard objects for $L(\lambda)$. Since $L(\lambda) \in \mathcal{U}$, we have $\mathcal{U} \neq \emptyset$. Let $U,U' \in \mathcal{U}$. The union $U \cup U'$ is the image of the natural morphism $U \oplus U' \to I(\lambda)$, so $U \cup U'$ is a quotient of $U \oplus U'$. We clearly have
\[
\lbrack U \oplus U' \colon L(\mu) \rbrack \leq \lbrack U \colon L(\mu) \rbrack + \lbrack U' \colon L(\mu) \rbrack \;,
\]
which implies that $U \cup U' \in \mathcal{U}$ as well, i.e. $\mathcal{U}$ is closed under finite unions. Since $\mathcal{C}$ is artinian, we conclude from Lemma~\ref{lem:comletelattice} that $\mathcal{U}$ has a unique maximal element $\nabla(\lambda)$. Dually, if $L(\lambda)$ admits a projective cover $P(\lambda)$, then $P(\lambda)$ admits a unique largest quotient which is a standard object for $L(\lambda)$.
\end{rem}

Recall that we do not necessarily assume that $\mathcal{C}$ has enough injectives. We denote by $\Ind \mathcal{C}$ the \word{ind-completion} of $\mathcal{C}$; see \cite[\S6, \S8.6]{KS} for details. The category $\Ind \mathcal{C}$ is abelian and there is an exact embedding of $\mathcal{C}$ into $\Ind \mathcal{C}$. Since $\mathcal{C}$ is essentially small, it follows from \cite[Theorem 8.6.5(vi)]{KS} that $\Ind \mathcal{C}$ is a Grothendieck category, hence $\Ind \mathcal{C}$ has enough injectives by \cite[Theorem 9.6.2]{KS}. For $X,Y \in \mathcal{C}$ let $\Ext_{\mathcal{C}}^n(X,Y)$ be the set of equivalence classes of $n$-term Yoneda extensions of $Y$ by $X$ as in \cite[III.3.2]{Verdier} together with the Baer sum of extensions as addition. Similarly, we can consider Ext-groups in $\Ind \mathcal{C}$. Since $\Ind \mathcal{C}$ has enough injectives, these groups are also given as usual via a right derived functor. The embedding of $\mathcal{C}$ into $\Ind \mathcal{C}$ induces a group morphism
\begin{equation}
	\Ext_{\mathcal{C}}^n(X,Y) \to \Ext_{\Ind \mathcal{C}}^n(X,Y) \;.
\end{equation}
It is proven in \cite{Kevin} that this morphism is an isomorphism; as noted there, this result can be derived from \cite[Theorem 15.3.1]{KS} in principle. The conclusion is that we have connecting $\delta$-morphisms between the Ext-groups in $\mathcal{C}$, as though $\mathcal{C}$ had enough injectives. We can now formulate Ext-vanishing assumptions between (co)standard objects.

\begin{ass} \label{ass_ext_vanishing}
	We assume that each $L(\lambda)$ has a costandard object~$\nabla(\lambda)$ and a standard object $\Delta(\lambda)$ such that the following condition holds for all $\lambda,\mu \in \Lambda$ and $0 \leq i \leq 2$:
\[
\Ext_\mathcal{C}^i(\Delta(\lambda), \nabla(\mu)) = \left\lbrace \begin{array}{ll} K & \tn{if } i=0 \tn{ and } \lambda = \mu \;, \\ 0 & \tn{else}  \;.\end{array} \right.
\]
\end{ass}

This assumption allows us to derive the following properties of standard and costandard objects, well-known from the theory of highest weight categories.

\begin{lem} \label{ext_lemma}
\hfill
\begin{enumerate}
	\item $\End_{\mathcal{C}}(L(\lambda)) \simeq K$, $\End_{\mathcal{C}}(\nabla(\lambda)) \simeq K$, and $\End_{\mathcal{C}}(\Delta(\lambda)) \simeq K$.
	\item If $\Ext_{\mathcal{C}}^1(L(\lambda),\nabla(\mu)) \neq 0$ or $\Ext_{\mathcal{C}}^1(\Delta(\mu),L(\lambda)) \neq 0$, then $\mu < \lambda$.
	\item If $\Ext_{\mathcal{C}}^1(L(\mu),L(\lambda)) \neq 0$, then $\mu < \lambda$ or $\lambda < \mu$.
	\item If $\Ext_{\mathcal{C}}^1(\nabla(\mu),\nabla(\lambda)) \neq 0$ or $\Ext_{\mathcal{C}}^1(\Delta(\mu),\Delta(\lambda)) \neq 0$, then $\mu > \lambda$.
\end{enumerate}
\end{lem}

\begin{proof}
We show that all statements involving $\nabla$ hold. The remaining statements involving $\Delta$ instead of $\nabla$ can be proven dually. We begin with part (2). Applying $\Hom_{\mathcal{C}}(-,\nabla(\mu))$ to the exact sequence
\begin{equation}
	0 \to \Rad \Delta(\lambda) \to \Delta(\lambda) \to L(\lambda) \to 0
\end{equation}
yields the exact sequence
\begin{equation} \label{ext_lemma_leq1}
\begin{aligned}
	0 & \to \Hom_{\mathcal{C}}(L(\lambda),\nabla(\mu)) \to \Hom_{\mathcal{C}}(\Delta(\lambda),\nabla(\mu)) \to \Hom_{\mathcal{C}}(\Rad \Delta(\lambda),\nabla(\mu)) \\
	& \to \Ext_{\mathcal{C}}^1(L(\lambda),\nabla(\mu)) \to \Ext_{\mathcal{C}}^1(\Delta(\lambda),\nabla(\mu)) = 0 \;.
\end{aligned}
\end{equation}
Because $\Hom_{\mathcal{C}}(L(\lambda),\nabla(\mu))$ injects into $\Hom_{\mathcal{C}}(\Delta(\lambda),\nabla(\mu))$ and the latter is equal to $\delta_{\lambda \mu} K$ by assumption, it follows that
\begin{equation} \label{ext_lemma_hom_l_nabla}
	\Hom_{\mathcal{C}}(L(\lambda),\nabla(\mu)) = \delta_{\lambda \mu} K \;,
\end{equation}
since $\dim_K \Hom_{\mathcal{C}}(L(\lambda),\nabla(\lambda)) \ge 1$. If $f \colon \Rad \Delta(\lambda) \to \nabla(\mu)$ is a non-zero morphism, then $\Image f$ is a non-zero subobject of $\nabla(\mu)$, hence it must contain $\Soc \nabla(\mu) \simeq L(\mu)$. It follows that $\Rad \Delta(\lambda)$ has $L(\mu)$ as a composition factor and hence $\mu < \lambda$. We conclude that:
\begin{equation}  \label{ext_hom_rad_delta}
	\Hom_{\mathcal{C}}(\Rad \Delta(\lambda), \nabla(\mu)) \neq 0 \Longrightarrow \mu < \lambda \;.
\end{equation}
It follows from \eqref{ext_lemma_leq1} that $\Hom_{\mathcal{C}}(\Rad \Delta(\lambda),\nabla(\mu))$ is isomorphic to $\Ext_{\mathcal{C}}^1(L(\lambda),\nabla(\mu))$. Hence, if $\lambda = \mu$, then $\Ext_{\mathcal{C}}^1(L(\lambda),\nabla(\mu)) = 0$ and we conclude that
\begin{equation} \label{ext_lemma_ext_l_nabla}
	\Ext_{\mathcal{C}}^1(L(\lambda),\nabla(\mu)) \neq 0 \Longrightarrow \mu < \lambda \;,
\end{equation}
showing that (2) holds.

Next, we show that statement (1) holds. Applying $\Hom_{\mathcal{C}}(L(\mu),-)$ to the exact sequence
\begin{equation}\label{ext_lemma_seqd}
	0 \to L(\lambda) \to \nabla(\lambda) \to \nabla(\lambda)/\Soc \nabla(\lambda) \to 0
\end{equation}
yields the exact sequence
\begin{equation} \label{ext_lemma_seq2}
\begin{aligned}
0 & \to \Hom_{\mathcal{C}}(L(\mu),L(\lambda)) \to \Hom_{\mathcal{C}}(L(\mu),\nabla(\lambda)) \to \Hom_{\mathcal{C}}(L(\mu),\nabla(\lambda)/\Soc \nabla(\lambda)) \\
& \to \Ext_{\mathcal{C}}^1(L(\mu),L(\lambda)) \to \Ext_{\mathcal{C}}^1(L(\mu),\nabla(\lambda)) \;.
\end{aligned}
\end{equation}
Since $\Hom_{\mathcal{C}}(L(\mu),L(\lambda))$ injects into $\Hom_{\mathcal{C}}(L(\mu),\nabla(\lambda))$ and the latter is equal to $\delta_{\mu \lambda} K$ by \eqref{ext_lemma_hom_l_nabla}, it follows that
\begin{equation}
	\Hom_{\mathcal{C}}(L(\mu),L(\lambda)) = \delta_{\mu \lambda} K \;.
\end{equation}
Applying instead $\Hom_{\mathcal{C}}(-,\nabla(\lambda))$ to the short exact sequence \eqref{ext_lemma_seqd} yields the exact sequence
\begin{equation}
	\begin{aligned}
		0 & \to \Hom_{\mathcal{C}}(\nabla(\lambda)/\Soc \nabla(\lambda), \nabla(\lambda)) \to \Hom_{\mathcal{C}}(\nabla(\lambda),\nabla(\lambda)) \\ & \to \Hom_{\mathcal{C}}(L(\lambda),\nabla(\lambda)) \;.
	\end{aligned}
\end{equation}
If $f \colon \nabla(\lambda)/\Soc \nabla(\lambda) \to \nabla(\lambda)$ is a non-zero morphism, then $\Image f$ must contain $\Soc \nabla(\lambda) \simeq L(\lambda)$, hence $L(\lambda)$ must be a composition factor of $\nabla(\lambda)/\Soc \nabla(\lambda) $ but this is a contradiction. We conclude that $\Hom_{\mathcal{C}}(\nabla(\lambda)/\Soc \nabla(\lambda), \nabla(\lambda)) = 0$ and therefore $\Hom_{\mathcal{C}}(\nabla(\lambda),\nabla(\lambda)) $ injects into $\Hom_{\mathcal{C}}(L(\lambda),\nabla(\lambda))$. The latter is equal to $K$ by \eqref{ext_lemma_hom_l_nabla}, hence
\begin{equation}
\End_{\mathcal{C}}(\nabla(\lambda)) = K,
\end{equation}
which completes the proof of statement (1).

Next, we show that statement (3) holds. By Definition~\ref{defn:(co)standard},
\begin{equation} \label{ext_lemma_hom_l_nabla_soc}
	\Hom_{\mathcal{C}}(L(\mu),\nabla(\lambda)/\Soc \nabla(\lambda)) \neq 0 \Longrightarrow \mu < \lambda \;.
\end{equation}
If $\mu \nless \lambda$ and $\lambda \nless \mu$, then by \eqref{ext_lemma_hom_l_nabla_soc} and \eqref{ext_lemma_ext_l_nabla} the terms to the left and right of $\Ext_{\mathcal{C}}^1(L(\mu),L(\lambda))$ in \eqref{ext_lemma_seq2} are zero and therefore $\Ext_{\mathcal{C}}^1(L(\mu),L(\lambda)) = 0$. Therefore
\begin{equation}
\Ext_{\mathcal{C}}^1(L(\mu),L(\lambda)) \neq 0 \Longrightarrow \mu < \lambda \text{ or } \lambda < \mu \;.
\end{equation}

Finally, we show that statement (4) holds. If $\Ext_{\mathcal{C}}^1(\nabla(\mu),\nabla(\lambda)) \neq 0$, then the   Lemma~\ref{lem:compfactorext1} below says that $\Ext_{\mathcal{C}}^1(L(\tau),\nabla(\lambda)) \neq 0$ for some composition factor $L(\tau)$ of $\nabla(\mu)$, hence $\mu \geq \tau > \lambda$ by \eqref{ext_lemma_ext_l_nabla}. We conclude that:
\begin{equation}
	\Ext_{\mathcal{C}}^1(\nabla(\mu),\nabla(\lambda)) \neq 0 \Longrightarrow \mu > \lambda \;,
\end{equation}
which is statement (4). This completes the proof of the lemma.
\end{proof}

The following standard fact was used in the proof of Lemma~\ref{ext_lemma}.

\begin{lem}\label{lem:compfactorext1}
	If $M$ and $N$ are objects such that $\Ext_{\mathcal{C}}^1(M,N) \neq 0$, then there is a composition factor $S$ of $M$ such that $\Ext_{\mathcal{C}}^1(S,N) \neq 0$.
\end{lem}

\begin{proof}
	Let $M_0 < M_1 < \ldots < M_n = M$ be a composition series of $M$ and let $S_i \coloneqq M_i/M_{i-1}$. We have a short exact sequence $0 \to M_{i-1} \to M_i \to S_i \to 0$ and this yields an exact sequence $\Ext_{\mathcal{C}}^1(S_i,N) \to \Ext_{\mathcal{C}}^1(M_i,N) \to \Ext_{\mathcal{C}}^1(M_{i-1},N)$. If $\Ext_{\mathcal{C}}^1(S_n,N) \neq 0$ then we are done. If $\Ext_{\mathcal{C}}^1(S_n,N) = 0$, then $\Ext_{\mathcal{C}}^1(M,N) \to \Ext_{\mathcal{C}}^1(M_{n-1},N)$ is injective and therefore $\Ext_{\mathcal{C}}^1(M_{n-1},N) \neq 0$. Inductively it follows that we must have $\Ext_{\mathcal{C}}^1(S_i,N) \neq 0$ for some $i$.
\end{proof}

\subsection{Tilting objects}

We say that an object $N \in \mathcal{C}$ has a \word{costandard filtration} if there is a filtration
\begin{equation} \label{costandard_filtration}
	0 = N_0 \subsetneq N_1 \subsetneq \ldots \subsetneq N_{n-1} \subsetneq N_n = N
\end{equation}
by subobjects $N_i$ of $N$ such that for each $i$ the quotient $N_i/N_{i-1}$ is isomorphic to a costandard object $\nabla(\lambda_i)$, for some $\lambda_i \in \Lambda$. Dually, one can define \word{standard filtrations}. We denote by $\mathcal{C}^\nabla$, respectively $\mathcal{C}^\Delta$, the full subcategory of $\mathcal{C}$ consisting of the objects admitting a costandard, respectively standard, filtration. Note that, since $\Ext_{\mathcal{C}}^1(\Delta(\lambda),\nabla(\mu)) = 0$ by \Autoref{ass_ext_vanishing}, we have
\begin{equation}
\Ext_{\mathcal{C}}^1(\Delta(\lambda),N) = 0 \text{ if } N \in \mathcal{C}^\nabla \quad \text{and} \quad \Ext_{\mathcal{C}}^1(M,\nabla(\mu)) = 0 \text{ if } M \in \mathcal{C}^\Delta \;.
\end{equation}

The following lemma is well-known:

\begin{lem} \hfill

\begin{enumerate}
\item If $N$ has a costandard filtration, then the number of times $\nabla(\lambda)$ occurs as a subquotient of a costandard filtration is independent of the costandard filtration and is given by
\begin{equation}
(N \colon \nabla(\lambda)) \coloneqq \dim_K \Hom_{\mathcal{C}} ( \Delta(\lambda), N ) \;.
\end{equation}
\item If $M$ has a standard filtration, then the number of times $\Delta(\lambda)$ occurs as a subquotient of a standard filtration is independent of the standard filtration and is given by
\begin{equation}
(M \colon \Delta(\lambda)) \coloneqq \dim_K \Hom_{\mathcal{C}} (M,\nabla(\lambda) ) \;.
\end{equation}
\end{enumerate}
\end{lem}

\begin{proof}
We only prove the first statement, the second is proven dually. The proof proceeds by induction on the length of the costandard filtration. Take a costandard filtration of $N$ as in \eqref{costandard_filtration}. If $n=1$, then $N=\nabla(\lambda_1)$. By \Autoref{ass_ext_vanishing} we have $\dim_K \Hom_{\mathcal{C}}(\Delta(\lambda),\nabla(\lambda_1)) = \delta_{\lambda, \lambda_1}$, hence the claim is true. Now, let $n>1$. Applying $\Hom(\Delta(\lambda),-)$ to the short exact sequence
\[
0 \to N_{n-1} \to N \to \nabla(\lambda_n) \to 0
\]
yields the exact sequence
\begin{align*}
0 & \to \Hom_{\mathcal{C}}(\Delta(\lambda), N_{n-1}) \to \Hom_{\mathcal{C}}(\Delta(\lambda),N) \to \Hom_{\mathcal{C}}(\Delta(\lambda), \nabla(\lambda_n)) \\ & \to \Ext_{\mathcal{C}}^1(\Delta(\lambda),N_{n-1}) \;.
\end{align*}
Since $N_{n-1}$ has a costandard filtration, it follows that $\Ext_{\mathcal{C}}^1(\Delta(\lambda),N_{n-1}) = 0$. Hence
\begin{align*}
	\dim_K \Hom_{\mathcal{C}}(\Delta(\lambda),N) &	= \dim_K \Hom_{\mathcal{C}}(\Delta(\lambda),\nabla(\lambda_n)) + \dim_K \Hom_{\mathcal{C}}(\Delta(\lambda),N_{n-1}) \\
	& = (\nabla(\lambda_n) \colon \nabla(\lambda)) + (N_{n-1} \colon \nabla(\lambda)) = (N \colon \nabla(\lambda)) \;,
\end{align*}
where we have used the induction assumption since $N_{n-1}$ and $\nabla(\lambda_n)$ have a costandard filtration of length $<n$.
\end{proof}

The proof of the following proposition can be found in \cite{AST, AST-Notes}. It uses the $\Ext^2$-vanishing from  \Autoref{ass_ext_vanishing}.

\begin{prop} \label{tilting_ext_vanish}
	The following holds for $M \in \mathcal{C}$:
	\begin{enumerate}
		\item $M \in \mathcal{C}^\Delta$ if and only if $\Ext_{\mathcal{C}}^1(M,\nabla(\lambda)) = 0$ for all $\lambda \in \Lambda$.
		\item $M \in \mathcal{C}^\nabla$ if and only if $\Ext_{\mathcal{C}}^1(\Delta(\lambda),M) = 0$ for all $\lambda \in \Lambda$.
	\end{enumerate}
\end{prop}

\begin{proof}
	We only show the first assertion of the proposition, the second is proven similarly.
	\Autoref{ass_ext_vanishing} immediately implies that if $M \in \mathcal{C}^\Delta$, then $\Ext_{\mathcal{C}}^1(M,\nabla(\lambda)) = 0$ for all $\lambda \in \Lambda$. The converse can be proven by exactly the same arguments as in \cite[Proposition 3.5]{AST-Notes} but we will include the details here to show where \Autoref{ass_ext_vanishing} is used. Let $\mathcal{X}$ be the class of all objects $M \in \mathcal{C}$ satisfying $\Ext_{\mathcal{C}}^1(M,\nabla(\lambda)) = 0$ for all $\lambda \in \Lambda$. Let $M \in \mathcal{X}$. Since $M$ has only finitely many composition factors, we can find $\lambda \in \Lambda$ minimal with the property that $\Hom_{\mathcal{C}}(M,L(\lambda)) \neq 0$. Let $\pi:\Delta(\lambda) \twoheadrightarrow L(\lambda)$ be the projection. We have $\mu < \lambda$ for all composition factors $L(\mu)$ of $\Ker \pi$. We claim that $\Ext_{\mathcal{C}}^1(M,\Ker \pi) = 0$. Suppose that $\Ext_{\mathcal{C}}^1(M,\Ker \pi) \neq 0$. Then there must be a composition factor $L(\mu)$ of $\Ker \pi$ such that $\Ext_{\mathcal{C}}^1(M,L(\mu)) \neq 0$. From the exact sequence $0 \rarr L(\mu) \rarr \nabla(\mu) \rarr \nabla(\mu)/L(\mu) \rarr 0$ we obtain the exact sequence
	\[
	\Hom_{\mathcal{C}}(M,\nabla(\mu)/L(\mu)) \rarr \Ext_{\mathcal{C}}^1(M,L(\mu)) \rarr \Ext_{\mathcal{C}}^1(M,\nabla(\mu)) = 0 \;.
	\]
	Since $\Ext_{\mathcal{C}}^1(M,L(\mu)) \neq 0$, we must have $\Hom_{\mathcal{C}}(M,\nabla(\mu)/L(\mu)) \neq 0$, so there is a non-zero morphism $\varphi:M \rarr \nabla(\mu)/L(\mu)$. We can now find a constituent $L(\nu)$ of $\Image \varphi$ and a non-zero morphism $\Image \varphi \rarr L(\nu)$ such that the composition with $\varphi$ yields a non-zero morphism $M \rarr L(\nu)$. Note that $\nu < \mu$. Since $\nu < \mu < \lambda$, this contradicts the minimality of $\lambda$. Hence, we have shown that $\Ext_{\mathcal{C}}^1(M,\Ker \pi) = 0$. If we now choose a non-zero $\varphi \in \Hom_{\mathcal{C}}(M,L(\lambda))$, then there is $\hat{\varphi} \in \Hom_{\mathcal{C}}(M,\Delta(\lambda))$ with $\pi \circ \hat{\varphi} = \varphi$. Since $\pi$ is essential and $\varphi$ is surjective, the map $\hat{\varphi}$ is also surjective. From the exact sequence $0 \rarr \Ker \hat{\varphi} \rarr M \rarr \Delta(\lambda) \rarr 0$ we obtain
	\[
	0 = \Ext_{\mathcal{C}}^1(M,\nabla(\mu)) \leftarrow \Ext_{\mathcal{C}}^1(\Ker \hat{\varphi}, \nabla(\mu)) \leftarrow \Ext_{\mathcal{C}}^2(\Delta(\lambda),\nabla(\mu)) = 0 \;,
	\]
	where the last equality follows from \Autoref{ass_ext_vanishing}, so $\Ker \hat{\varphi} \in \mathcal{X}$. We can now argue by induction on the length of $M \in \mathcal{X}$ that all objects in $\mathcal{X}$ admit a standard filtration. Namely, let $M \in \mathcal{X}$ be of minimal length. Then, since $\Ker \hat{\varphi} \in \mathcal{X}$, we must have $\Ker \hat{\varphi} = 0$, so $M \simeq \Delta(\lambda)$, showing that $M$ admits a standard filtration. If $M$ has arbitrary length, we know that $M/\Ker \hat{\varphi} \simeq \Delta(\lambda)$ and by induction that $\Ker \hat{\varphi}$ admits a standard filtration, so $M$ also admits a standard filtration.
\end{proof}

An object $T \in \mathcal{C}$ admitting both a standard and a costandard filtration is called a \word{tilting} object. We denote by $\mathcal{C}^t \coloneqq \mathcal{C}^\Delta \cap \mathcal{C}^\nabla$ the full subcategory of $\mathcal{C}$ consisting of tilting objects.

\begin{cor} \label{tilting_krull_schmidt}
	The following holds:
	\begin{enumerate}
		\item \label{tilting_krull_schmidt:ext} If $T \in \mathcal{C}$, then $T \in \mathcal{C}^t$ if and only if $\Ext_{\mathcal{C}}^1(T,\nabla(\lambda)) = 0 = \Ext_{\mathcal{C}}^1(\Delta(\lambda),T)$ for all $\lambda \in \Lambda$.
		\item \label{tilting_krull_schmidt:closed} $\mathcal{C}^t$ is closed under direct sums and direct summands in $\mathcal{C}$.
		\item $\mathcal{C}^t$ is a Krull–Schmidt category.
	\end{enumerate}
\end{cor}

\begin{proof}
	The first assertion is just \Autoref{tilting_ext_vanish}. Part \ref{tilting_krull_schmidt:ext} clearly implies part \ref{tilting_krull_schmidt:closed} since $\Ext_\mathcal{C}^1$ is a bi-additive functor. Since $\mathcal{C}$ is locally finite by assumption, it is Krull–Schmidt by \cite[Lemma 5.1, Theorem 5.5]{Krause-KS}. Hence, $\mathcal{C}^t$ is also Krull–Schmidt by part \ref{tilting_krull_schmidt:closed}.
\end{proof}

Since $\mathcal{C}^t$ is Krull–Schmidt, every tilting object is a finite direct sum  of indecomposable tilting objects, the decomposition being unique up to permutation and isomorphism of the summands. Our final assumption on $\mathcal{C}$ is that the indecomposable tilting objects behave as in Ringel's theory \cite{Ringel-Filtrations}.

\begin{ass} \label{ass_tilting}
	We assume:
	\begin{enumerate}
		\item For any $\lambda \in \Lambda$ there is an indecomposable object $T(\lambda) \in \mathcal{C}^t$ such that:
		\begin{enumerate}
			\item \label{ass_tilting:cond1} $T(\lambda)$ is of highest weight $\lambda$, i.e. if $\lbrack T(\lambda) : L(\mu) \rbrack \neq 0$, then $\mu \leq \lambda$, and $\lbrack T(\lambda) : L(\lambda) \rbrack = 1$;
			\item \label{ass_tilting:cond2} there is a monomorphism $\Delta(\lambda) \hookrightarrow T(\lambda)$;
			\item \label{ass_tilting:cond3} there is an epimorphism $T(\lambda) \twoheadrightarrow \nabla(\lambda)$.
		\end{enumerate}
		\item The map $\lambda \mapsto T(\lambda)$ is a bijection between $\Lambda$ and the set of isomorphism classes of indecomposable tilting objects of $\mathcal{C}$.
	\end{enumerate}
\end{ass}

\begin{rem}
Assuming conditions (\ref{ass_tilting:cond2}) and (\ref{ass_tilting:cond3}) in \Autoref{ass_tilting}, condition (\ref{ass_tilting:cond1}) can be replaced by
\begin{equation} \label{ass_tilting_cond1_delta}
	\text{if } (T(\lambda) \colon \Delta(\mu)) \neq 0\text{, then } \mu \leq \lambda\text{, and } (T(\lambda) \colon \Delta(\lambda)) = 1,
\end{equation}
or by
\begin{equation} \label{ass_tilting_cond1_nabla}
	\text{if } (T(\lambda) \colon \nabla(\mu)) \neq 0\text{, then } \mu \leq \lambda\text{, and } (T(\lambda) \colon \nabla(\lambda)) = 1.
\end{equation}
Namely, suppose that (\ref{ass_tilting:cond1}) holds. If $(T(\lambda) \colon \Delta(\mu)) \neq 0$, then clearly $\lbrack T(\lambda) \colon L(\mu) \rbrack \neq 0$, hence, $\mu \leq \lambda$. By assumption, there is a monomorphism $\Delta(\lambda) \hookrightarrow T(\lambda)$, hence $(T(\lambda) \colon \Delta(\lambda)) \neq 0$. If $(T(\lambda) \colon \Delta(\lambda)) > 1$, then also $\lbrack T(\lambda) \colon L(\lambda) \rbrack > 1$, which is a contradiction. Hence, $(T(\lambda) \colon \Delta(\lambda) ) = 1$. Conversely, suppose that \eqref{ass_tilting_cond1_delta} holds. If $\lbrack T(\lambda) \colon L(\mu) \rbrack \neq 0$, there is $\nu \in \Lambda$ such that $(T(\lambda) \colon \Delta(\nu)) \neq 0$ and $\lbrack \Delta(\nu) \colon L(\mu) \rbrack \neq 0$, hence $\mu \leq \nu \leq \lambda$. Since $(T(\lambda) \colon \Delta(\lambda)) \neq 0$, we have $\lbrack T(\lambda) \colon L(\lambda) \rbrack \neq 0$. Any occurrence of $L(\lambda)$ as a composition factor of $T(\lambda)$ must be in $\Delta(\lambda)$ by the properties of standard objects, and since $L(\lambda)$ occurs precisely once in $\Delta(\lambda)$ and $\Delta(\lambda)$ occurs precisely once in $T(\lambda)$, we conclude that $\lbrack T(\lambda) \colon L(\lambda) \rbrack = 1$. Similarly, one shows that (\ref{ass_tilting:cond1}) is equivalent to \eqref{ass_tilting_cond1_nabla}.
\end{rem}

\begin{prop} \label{tilting_standard_homs}
	The following holds:
	\begin{enumerate}
		\item If $\Hom_\mathcal{C}(\Delta(\lambda),T(\mu)) \neq 0$, then $\lambda \leq \mu$. Moreover, $\Hom_\mathcal{C}(\Delta(\lambda),T(\lambda)) \simeq K$, so the embedding $\Delta(\lambda) \hookrightarrow T(\lambda)$ is unique up to scalars.
		\item \label{tilting_standard_homs:nabla} If $\Hom_\mathcal{C}(T(\mu),\nabla(\lambda)) \neq 0$, then $\lambda \leq \mu$. Moreover, $\Hom_\mathcal{C}(T(\lambda),\nabla(\lambda)) \simeq K$, so the projection $T(\lambda) \twoheadrightarrow \nabla(\lambda)$ is unique up to scalars.
		\item The composition $\Delta(\lambda) \hookrightarrow T(\lambda) \twoheadrightarrow \nabla(\lambda)$ is non-zero.
	\end{enumerate}
\end{prop}

\begin{proof}
	Let $\varphi:\Delta(\lambda) \rarr T(\mu)$ be a non-zero morphism. This induces an isomorphism $\Delta(\lambda)/\Ker \varphi \simeq \Image \varphi$. Since $\varphi$ is non-zero, we have $\Ker \varphi \subsetneq \Delta(\lambda)$, so $\Ker \varphi \subs \Rad \Delta(\lambda)$, implying that $\lbrack \Delta(\lambda)/\Ker \varphi : L(\lambda) \rbrack \neq 0$. Hence, $\lbrack T(\mu):L(\lambda) \rbrack \neq 0$, so $\lambda \leq \mu$. Choose an embedding $i:\Delta(\lambda) \hookrightarrow T(\lambda)$ and let $q:T(\lambda) \twoheadrightarrow T(\lambda)/\Delta(\lambda)$ be the quotient map, where we identify $\Delta(\lambda)$ with $\Image i$. We get an exact sequence
	\[
	\begin{tikzcd}[column sep=small]
	0 \arrow{r} & \Delta(\lambda) \arrow{r}{i} & T(\lambda) \arrow{r}{q} & T(\lambda)/\Delta(\lambda) \arrow{r} & 0 \;.
	\end{tikzcd}
	\]
	Applying $\Hom_{\mathcal{C}}(\Delta(\lambda),-)$ yields an exact sequence
	\begin{align*}
	0 \to & \Hom_{\mathcal{C}}(\Delta(\lambda),\Delta(\lambda)) \to \Hom_{\mathcal{C}}(\Delta(\lambda),T(\lambda)) \to \Hom_{\mathcal{C}}(\Delta(\lambda),T(\lambda)/\Delta(\lambda))  \\
	\to & \Ext_{\mathcal{C}}^1(\Delta(\lambda),\Delta(\lambda)) = 0 \;,
	\end{align*}
	where we have used \Autoref{ext_lemma}. Hence, $\Hom_{\mathcal{C}}(\Delta(\lambda),T(\lambda)/\Delta(\lambda))$ is the image of $\Hom_{\mathcal{C}}(\Delta(\lambda),T(\lambda))$ under composition with $q$. But it is clear that this image is equal to zero, so the above sequence yields
	\[
	\Hom_{\mathcal{C}}(\Delta(\lambda),T(\lambda)) \simeq \Hom_{\mathcal{C}}(\Delta(\lambda),\Delta(\lambda)) \simeq K \;,
	\]
	where we have used \Autoref{ext_lemma}. Part (\ref{tilting_standard_homs:nabla}) is proven similarly. Finally, suppose that $i:\Delta(\lambda) \hookrightarrow T(\lambda)$ is an embedding and $\pi:T(\lambda) \twoheadrightarrow \nabla(\lambda)$ is a projection such that the composition is zero. Then $\Image i \subs \Ker \pi$, so we get a non-zero epimorphism $T(\lambda)/\Delta(\lambda) \twoheadrightarrow \nabla(\lambda)$. Since $L(\lambda)$ is a constituent of $\Delta(\lambda)$ and of $\nabla(\lambda)$, we deduce that $\lbrack T(\lambda):L(\lambda) \rbrack \geq 2$. This is not possible, so $\pi \circ i \neq 0$.
\end{proof}

\subsection{Standard categories and examples}

\begin{defn} \label{defn_std_duality}
A \word{standard category} is a category satisfying Assumptions \ref{ass_lf}, \ref{ass_simples}, \ref{ass_ext_vanishing}, and \ref{ass_tilting}.
\end{defn}

We mention two general, and two specific, examples of standard categories.

\begin{ex} \label{finite_hwc_has_everything}
Let $\mathcal{C}$ be a highest weight category over a field $K$, in the sense of Cline--Parshall--Scott \cite{CPS}. Assume moreover that:
\begin{enumerate}
	\item all objects of $\mathcal{C}$ are of finite length;
	\item $\mathcal{C}$ has only \emph{finitely} many simple objects up to isomorphism;
	\item $\mathcal{C}$ is \word{split}, i.e. $\End_{\mathcal{C}}(L) \simeq K$  for any simple object $L$ of $\mathcal{C}$.
\end{enumerate}
Then $\mathcal{C}$ is a standard category. First, it follows from \cite[Theorem 3.6]{CPS} that $\mathcal{C}$ is equivalent to the category of finite-dimensional modules over a finite-dimensional $K$-algebra. This implies, in particular, that $\mathcal{C}$ is essentially small and locally finite, so \Autoref{ass_lf} holds. By definition of highest weight category, \Autoref{ass_simples} holds and there are costandard objects. In \cite[\S1A]{CPS-Duality-in-highest-weight-89} it is shown that $\mathcal{C}^{\mathrm{op}}$ is a highest weight category as well, hence $\mathcal{C}$ also has standard objects; note that what in \cite{CPS-Duality-in-highest-weight-89} is called an ``artinian'' category is what we call a locally finite category. In the proof of \cite[Theorem 3.11]{CPS} it is shown that the Ext-vanishing from \Autoref{ass_ext_vanishing} holds (this relies on the assumption that $\mathcal{C}$ is split). Finally, it is a result by Ringel \cite[Proposition 2]{Ringel-Filtrations} that tilting objects, as in \Autoref{ass_tilting}, exist.
\end{ex}

\begin{ex} \label{lfhw_has_everything}
Let $\mathcal{C}$ be a lower finite highest weight category in the sense of Brundan and Stroppel \cite[Definition 3.50]{BrundanStroppel}. Then $\mathcal{C}$ is a standard category. The Ext-vanishing in \Autoref{ass_ext_vanishing} follows from \cite[Theorem~3.56]{BrundanStroppel} and the existence of tilting objects, \Autoref{ass_tilting}, follows from \cite[Theorem 4.2]{BrundanStroppel}.
\end{ex}

%more directly from the discussion in \cite[\S3]{Kevin}. Namely, $\mathcal{C}$ is an essentially small and locally finite abelian category by definition, hence satisfies \Autoref{ass_lf}. That $\mathcal{C}$ has both standard and costandard objects follows from the equivalence theorem \cite[Theorem 3.1.3]{Kevin}. The precise assumptions about the socle and radical follow from the fact that $\nabla(\lambda)$, respectively $\Delta(\lambda)$, is the injective hull, respectively projective cover, of $L(\lambda)$ in the truncation $\mathcal{C}_{\leq \lambda}$. The Ext-vanishing in \Autoref{ass_ext_vanishing} follows from \cite[\S3]{Kevin}.

\begin{ex}
The prototypical example of a standard category is the Bernstein--Gelfand--Gelfand category $\mathcal{O}$ of a complex finite-dimensional semisimple Lie algebra. Proofs of all the assumptions making $\mathcal{O}$ a standard category can be found in \cite{HumphreysBGG}.
\end{ex}

\begin{ex} \label{uq_mod_examples}
Let $\Phi$ be a root system with Cartan matrix $A$ and let $q$ be a non-zero element of a field $K$. If $q$ is a root of unity, we assume that $q$ is a primitive root of unity of odd order $l$; and if $\Phi$ has an irreducible component of type $G_2$, we moreover assume that $l$ is prime to $3$. Let $U_q$ be the quantum group associated to~$A$ and specialized in $q$ (by this we mean the specialization of Lusztig's integral form \cite{Lusztig}, which involves divided power generators). Let $\cat{mod}{U_q}$ be the category of type-1 finite-dimensional $U_q$-modules; see \cite{APW}. Then $\cat{mod}{U_q}$ is a standard category with indexing set $\Lambda$ being the set of dominant integral weights of $\Phi$. This follows from the detailed discussion in \cite{AST,AST-Notes}. We note that $\cat{mod}{U_q}$ does not necessarily have enough injectives—and is therefore not a highest weight category—if $K$ is of positive characteristic.
\end{ex}

\section{Generalizing the Andersen--Stroppel--Tubbenhauer construction} \label{cellularity_section}

In this section we will describe how to generalize the work of Andersen, Stroppel, and Tubbenhauer \cite{AST} from tilting modules of quantum groups, as in \Autoref{uq_mod_examples}, to tilting objects in an arbitrary standard category $\mathcal{C}$. As noted in the introduction, the key limitation in \cite{AST} is the reliance on weight spaces of modules. We replace this with a categorical filtration on the Hom-spaces, see  \eqref{hom_filtration}. Once we have established the ``basis theorem'' (\Autoref{cellular_newsection_basis_independent}), the proof of the main Theorems \ref{end_tilting_is_cellular} and \ref{cellularity_thm_precise} is exactly as in \cite{AST}. But for the proof of \Autoref{cellular_newsection_basis_independent} we need to argue differently. Since this concerns several technicalities from \cite{AST}, we have to repeat some of the constructions and arguments from \cite{AST}. In \Autoref{subsec_duality} we discuss dualities and the problem of when they actually induce an involution on the endomorphism algebra of tilting objects—this led us to the notion of ``standard dualities''. \\

Throughout, $\mathcal{C}$ denotes a standard category over a field $K$.

\subsection{Basic construction} \label{sec_basic_construction}

Let us choose for any $\lambda \in \Lambda$ a non-zero morphism
\begin{equation} \label{c_lambda_def}
c^\lambda: \Delta(\lambda) \rarr \nabla(\lambda)\;.
\end{equation}
By \Autoref{ass_ext_vanishing}, this is unique up to scalar. By \Autoref{tilting_standard_homs} we can choose an embedding
\begin{equation}
  i^\lambda:\Delta(\lambda) \hookrightarrow T(\lambda)
\end{equation}
and a  projection
\begin{equation} \label{pi_map}
  \pi^\lambda:T(\lambda) \twoheadrightarrow \nabla(\lambda)
\end{equation}
 so that we get a factorization
\begin{equation} \label{c_map}
c^\lambda = \pi^\lambda \circ i^\lambda \;.
\end{equation}

\begin{lem} \label{tilting_lift_lemma}
	Let $\lambda \in \Lambda$. The following holds:
	\begin{enumerate}
		\item Let $M \in \mathcal{C}^\Delta$. Then any morphism $f:M \rarr \nabla(\lambda)$ factors through a morphism $\hat{f}:M \rarr T(\lambda)$, i.e., there is a commutative diagram
		\[
		\begin{tikzcd}
		M \arrow{r}{\hat{f}} \arrow{dr}[swap]{f}  & T(\lambda)\arrow[twoheadrightarrow]{d}{\pi^\lambda} \\
		& \nabla(\lambda) \; .
		\end{tikzcd}
		\]
		\item Let $N \in \mathcal{C}^\nabla$. Then any morphism $g:\Delta(\lambda) \rarr N$ extends to a morphism $\hat{g}:T(\lambda) \rarr N$, i.e., there is a commutative diagram
		\[
		\begin{tikzcd}
		T(\lambda) \arrow{r}{\hat{g}} & N \\
		\Delta(\lambda) \arrow[hookrightarrow]{u}{i^\lambda} \arrow{ur}[swap]{g} \; .
		\end{tikzcd}
		\]
	\end{enumerate}
\end{lem}

\begin{proof}
	We consider the first statement, the second is dual. Since $\nabla(\lambda)$ occurs at the top of a costandard filtration of $T(\lambda)$, it is clear that $\Ker \pi^\lambda$ has a costandard filtration and so it follows from \Autoref{tilting_ext_vanish} that $\Ext_\mathcal{C}^1(M, \Ker \pi^\lambda) = 0$ since $M$ has a standard filtration. Applying $\Hom_\mathcal{C}(M,-)$ to the exact sequence $0 \rarr \Ker \pi^\lambda \rarr T(\lambda) \overset{\pi^\lambda}{\rarr}\nabla(\lambda) \rarr 0$ proves the claim.
\end{proof}

In the following we let $M \in \mathcal{C}^\Delta$ and $N \in \mathcal{C}^\nabla$. For a $K$-basis $F_M^\lambda$ of $\Hom_\mathcal{C}(M,\nabla(\lambda))$ and a choice of lift $\hat{f}$ (as in  \Autoref{tilting_lift_lemma}) for each $f \in F_M^\lambda$ we set
\begin{equation}
\hat{F}_M^\lambda \coloneqq \lbrace \hat{f} \mid f \in F_M^\lambda \rbrace \subs \Hom_\mathcal{C}(M,T(\lambda)) \;.
\end{equation}
Similarly, for a $K$-basis $G_N^\lambda$ of $\Hom_\mathcal{C}(\Delta(\lambda),N)$  and a choice of lifts $\hat{g}$ we set
\begin{equation}
\hat{G}_N^\lambda \coloneqq \lbrace \hat{g} \mid g \in G_N^\lambda \rbrace \subs \Hom_\mathcal{C}(T(\lambda),N)\;.
\end{equation}
Note that the lifts, and thus the subsets $\hat{F}_M^\lambda$ and $\hat{G}_N^\lambda$, are not unique. For any such choice, we define the subset
\begin{equation}
\hat{G}_N^\lambda \hat{F}_M^\lambda \coloneqq \lbrace \hat{g} \circ \hat{f} \mid \hat{g} \in \hat{G}_N^\lambda, \hat{f} \in \hat{F}_M^\lambda \rbrace \subs \Hom_\mathcal{C}(M,N) \;.
\end{equation}
The elements of this subset can be illustrated by the commutative diagram
\begin{equation}
\begin{tikzcd}
& \Delta(\lambda) \arrow[hookrightarrow]{d}[swap]{i^\lambda} \arrow{dr}{g} \\
M \arrow{r}{\hat{f}} \arrow{dr}[swap]{f} & T(\lambda) \arrow{r}[swap]{\hat{g}} \arrow[twoheadrightarrow]{d}{\pi^\lambda} & N \\
& \nabla(\lambda)
\end{tikzcd}
\end{equation}
Let
\begin{equation}
\hat{G}_N \hat{F}_M \coloneqq \bigcup_{\lambda \in \Lambda} \hat{G}_N^\lambda \hat{F}_M^\lambda \;.
\end{equation}
It will follow from Proposition \ref{cellular_newsection_ingredients} that the above union is disjoint. We define indexing sets
\begin{equation}
\mathcal{I}_N^\lambda \coloneqq \{ 1,\ldots, \lbrack N:\nabla(\lambda) \rbrack \} \quad \text{and} \quad \mathcal{J}_M^\lambda \coloneqq \{ 1,\ldots, \lbrack M:\Delta(\lambda) \rbrack \} \;,
\end{equation}
and write
\begin{equation}
F_M^\lambda = \{ f_j^\lambda \mid j \in \mathcal{J}_M^\lambda \} \quad \text{and} \quad G_N^\lambda = \{ g_i^\lambda \mid i \in \mathcal{I}_N^\lambda \} \;.
\end{equation}
Hence, setting
\begin{equation}
	c_{ij}^\lambda \coloneqq \hat{g}_i^\lambda \circ \hat{f}_j^\lambda
\end{equation}
we have
\begin{equation}
\hat{G}_N  \hat{F}_M = \{ c_{ij}^\lambda \mid \lambda \in \Lambda, i \in \mathcal{I}_M^\lambda, j \in \mathcal{J}_N^\lambda \} \;.
\end{equation}

\subsection{The basis theorem}

Our aim, following \cite{AST}, is to prove the following basis theorem; see also \cite[Theorem~4.43]{BrundanStroppel}.

\begin{thm} \label{cellular_newsection_basis_independent}
	For any choice of lifts $\hat{f}$ and $\hat{g}$, the set $\hat{G}_N \hat{F}_M$ is a $K$-vector space basis of $\Hom_{\mathcal{C}}(M,N)$.
\end{thm}

The proof of the theorem, which is \cite[Theorem 3.1]{AST}, needs some preparation and some modifications to make it work in our general setting. The key point is the independence of the choice of bases and lifts. This is done in \cite{AST} by using a filtration of the Hom-spaces given by restrictions of morphisms to weight spaces. We will give an alternative proof of these facts, not relying on the notion of weight spaces. For $\lambda \in \Lambda$ and $\varphi \in \Hom_{\mathcal{C}}(M,N)$, we define
\begin{equation} \label{cellular_newsection_key_modification}
\varphi_\lambda \coloneqq \lbrack \Image \varphi:L(\lambda) \rbrack  \in \bbN \;.
\end{equation}
%Our $\varphi_\lambda$, being a natural number, is very different from the $\varphi_\lambda$ in \cite{AST}. Nonetheless, the following set can then be defined similarly as in \cite{AST}:
Then,
\begin{equation} \label{hom_filtration}
\Hom_{\mathcal{C}}(M,N)^{\leq \lambda} \coloneqq \lbrace \varphi \in \Hom_\mathcal{C}(M,N) \mid \varphi_\mu = 0 \tn{ unless } \mu \leq \lambda \rbrace \;.
\end{equation}
Analogously, we define $\Hom_{\mathcal{C}}(M,N)^{< \lambda}$.

\begin{lem} \label{hom_weight_space_is_space}
	The set $\Hom_{\mc{C}}(M,N)^{\leq \lambda}$ is a vector subspace of $\Hom_{\mc{C}}(M,N)$.
\end{lem}

\begin{proof}
	Let $\phi,\psi \in \Hom_{\mc{C}}(M,N)^{\leq \lambda}$. Since $\Image \phi, \Image \psi, \Image(\phi+\psi) \subs \Image \phi+ \Image \psi$, we have
	\[
	(\phi+ \psi)_\mu = \lbrack \Image(\phi+\psi):L(\mu) \rbrack \leq \lbrack \Image \phi : L(\mu) \rbrack + \lbrack \Image \psi : L(\mu) \rbrack = \phi_\mu + \psi_\mu \;.
	\]
	If $\mu \not\leq \lambda$, then $\phi_\mu = 0 = \psi_\mu$ since $\phi,\psi \in \Hom_\mc{C}(M,N)^{\leq \lambda}$. Hence $(\phi+\psi)_\mu = 0$, implying that $\phi+\psi \in \Hom_{\mc{C}}(M,N)^{\leq \lambda}$.
\end{proof}

\begin{lem} \label{f_plus_g_weight}
	If $f,g \in \Hom_\mc{C}(M,N)$ and $f_\lambda = 0$, then $(f+g)_\lambda = g_\lambda$.
\end{lem}

\begin{proof}
	Let $\ol{f+g}:M \rarr N/\Image(f)$ and $\ol{g}:M \rarr N/\Image(f)$ be the map induced by $f+g$ and $g$, respectively. Clearly, $\ol{f+g}=\ol{g}$. Since $0=f_\lambda = \lbrack \Image(f):L(\lambda) \rbrack$, we have $(f+g)_\lambda = (\ol{f+g})_\lambda = \ol{g}_\lambda = g_\lambda$.
\end{proof}

The following proposition contains the key ingredients for the proof of  \Autoref{cellular_newsection_basis_independent}. The first three statements are counterparts of \cite[Proposition 3.3]{AST}, \cite[Lemma 3.4]{AST}, and \cite[Lemma 3.5]{AST}. \Autoref{cellular_newsection_basis_independent} follows almost immediately from the third statement in the proposition. The third statement in turn relies crucially on the second statement, and this is a sufficient version of \cite[Lemma 3.4]{AST} that we can prove in our generalized setting. Since our $\varphi_\lambda$ is quite different from \cite{AST}, we will repeat some of the arguments in \cite{AST} to prove that our generalization indeed works.

\begin{prop} \label{cellular_newsection_ingredients}
For each $\lambda \in \Lambda$, choose arbitrary $F_M^\lambda$ and $\hat{F}_M^\lambda$. Then there is a choice of $G_N^\lambda$ and $\hat{G}_N^\lambda$, for all $\lambda \in \Lambda$, such that:
	\begin{enumerate}
		\item \label{cellular_newsection_ingredients:basis} $\hat{G}_N\hat{F}_M$ is a $K$-basis of $\Hom_\mc{C}(M,N)$.
		\item \label{cellular_newsection_ingredients:lemma45} If $\phi$ is a non-zero element of the $K$-span $\langle \hat{G}_N^\lambda \hat{F}_M^\lambda \rangle_K$ of $\hat{G}_N^\lambda \hat{F}_M^\lambda$, then $\phi_\lambda$ is non-zero.
	\end{enumerate}
	Moreover, any such choice satisfies the following property:
	\begin{enumerate}
		\setcounter{enumi}{2}
		\item \label{cellular_newsection_ingredients:filtered_basis}The set $(\hat{G}_N\hat{F}_M)^{\leq \lambda} \coloneqq \bigcup_{\mu \leq \lambda} \hat{G}_N^\mu \hat{F}_M^\mu$ is a $K$-basis of $\Hom_\mc{C}(M,N)^{\leq \lambda}$. Similarly, $(\hat{G}_N\hat{F}_M)^{< \lambda} $ is a basis of $\Hom_\mc{C}(M,N)^{< \lambda}$.
	\end{enumerate}
  We furthermore conclude:
  \begin{enumerate}
		\setcounter{enumi}{3}
    \item \label{cellular_newsection_ingredients:disjoint} The sets $\hat{G}_N^\lambda\hat{F}_M^\lambda$ for the various $\lambda \in \Lambda$ are pairwise disjoint.
  \end{enumerate}
\end{prop}

Note that Theorem \ref{cellular_newsection_basis_independent} is a strictly stronger statement than Proposition~\ref{cellular_newsection_ingredients}(\ref{cellular_newsection_ingredients:basis}) because the latter only claims that there is a choice of $G_N^\lambda$ and $\hat{G}_N^\lambda$ giving the desired basis, whereas the theorem says this is true for any choice.

\begin{proof}
	We will show (\ref{cellular_newsection_ingredients:basis}) and (\ref{cellular_newsection_ingredients:lemma45}) by induction on the length of a costandard filtration of $N$.

	First, assume that $N =\nabla(\lambda)$ for some $\lambda \in \Lambda$. Recall that $c^\lambda \colon \Delta(\lambda) \to \nabla(\lambda)$ from \eqref{c_map} is a basis element of $\Hom_{\mathcal{C}}(\Delta(\lambda),\nabla(\lambda))$. Let $g^\lambda \coloneqq c^\lambda$ and $G_N^\lambda \coloneqq \{ g^\lambda \}$. By construction, $\hat{g}^\lambda \coloneqq \pi^\lambda \colon T(\lambda) \to \nabla(\lambda)$ from \eqref{pi_map} is a lift of $g^\lambda$. We thus set $\hat{G}_N^\lambda \coloneqq \{ \hat{g}^\lambda \}$. By assumption, $F_M^\lambda$ is a basis of $\Hom_{\mathcal{C}}(M,\nabla(\lambda))$. For $f^\lambda \in F_M^\lambda$ with chosen lift $\hat{f}^\lambda \in \hat{F}_M^\lambda$ we have $\hat{g}^\lambda \circ \hat{f}^\lambda = \pi^\lambda \circ \hat{f}^\lambda = f^\lambda$ by the property of the lift. It follows that $\hat{G}_N^\lambda \hat{F}_M^\lambda$ is a basis of $\Hom_{\mathcal{C}}(M,N)$, proving  (\ref{cellular_newsection_ingredients:basis}). To prove (\ref{cellular_newsection_ingredients:lemma45}), we argue that $\hat{G}_N^\mu \hat{F}_M^\mu = \emptyset$ if $\mu \neq \lambda$. By definition, an element $\phi$ of $\hat{G}_N^\mu \hat{F}_M^\mu$  factorizes as $\phi = \hat{g} \hat{f}$ for some $\hat{f} \in \hat{F}_M^\mu$ and $\hat{g} \in \hat{G}_N^\mu$. Moreover, $\hat{g}$ is a lift of a morphism $g \in \Hom_\mc{C}(\Delta(\mu),\nabla(\lambda))$. By \Autoref{ass_ext_vanishing} we have $\Hom_\mc{C}(\Delta(\mu),\nabla(\lambda)) = 0$ whenever $\mu \neq \lambda$. Hence, if $\mu \neq \lambda$, then $G_N^\mu = \emptyset$, thus $\hat{G}_N^\mu = \emptyset$ and $\hat{G}_N^\mu \hat{F}_M^\mu = \emptyset$. Consequently, if $\phi \in \langle \hat{G}_N^\mu \hat{F}_M^\mu \rangle_K$ is non-zero, we must have $\mu = \lambda$. In this case, $\phi$ is a non-zero morphism $M \rarr \nabla(\lambda)=N$. The image is a non-zero submodule of $\nabla(\lambda)$, thus contains $\Soc \nabla(\lambda) = L(\lambda)$, hence $\phi_\lambda \neq 0$. This concludes the proof of (\ref{cellular_newsection_ingredients:basis}) and (\ref{cellular_newsection_ingredients:lemma45}) for $N=\nabla(\lambda)$.

	Now, let $N$ be arbitrary with costandard filtration $0 = N_0 \subset N_1 \subset \dots \subset N_{k-1} \subset N_k = N$. Let $\lambda \in \Lambda$ with $N_k/N_{k-1} \simeq \nabla(\lambda)$. Applying $\Hom_\mc{C}(M,-)$ to the exact sequence
	\begin{equation} \label{cellularity_hom_exact_sequence_2}
	0 \rarr N_{k-1} \rarr N \overset{\pi}{\rarr} \nabla(\lambda) \rarr 0
	\end{equation}
	yields the exact sequence
	\begin{equation} \label{cellularity_hom_exact_sequence}
	0 \rarr \Hom_\mc{C}(M,N_{k-1}) \overset{\mrm{inc}}{\rarr} \Hom_\mc{C}(M,N) \rarr \Hom_\mc{C}(M,\nabla(\lambda)) \rarr 0 \;,
	\end{equation}
	the exactness on the right hand side following from the fact that $M \in \mc{C}^\Delta$, $N_{k-1} \in \mc{C}^\nabla$, and \Autoref{ass_ext_vanishing}. By the induction assumption, we can choose a lift $\hat{G}_{N_{k-1}}$ such that the basis $\hat{G}_{N_{k-1}} \hat{F}_{M}$ of $\Hom_\mc{C}(M,N_{k-1})$ satisfies both (\ref{cellular_newsection_ingredients:basis}) and (\ref{cellular_newsection_ingredients:lemma45}). We can argue as in the second half of the proof of \cite[Proposition 3.3]{AST} that we can choose a suitable lift $\hat{G}_N$ such that the basis $\hat{G}_N\hat{F}_M = \bigcup_{\mu \in \Lambda} \hat{G}_N^\mu \hat{F}_M^\mu$ of $\Hom_\mc{C}(M,N)$ satisfies (\ref{cellular_newsection_ingredients:lemma45}). The choice of  $\hat{G}_N^\mu$ will depend on whether $\mu \neq \lambda$ or $\mu = \lambda$.

	If $\mu \neq \lambda$, we set $G_N^\mu \coloneqq \mrm{inc}\left(G_{N_{k-1}}^\mu\right)$ and $\hat{G}_N^\mu \coloneqq \mrm{inc}\left (\hat{G}_{N_{k-1}}^\mu\right)$. If $\mu = \lambda$, then $\mrm{inc}\left(G_{N_{k-1}}^\lambda\right)$ is still linearly independent but since $\lbrack N:\nabla(\lambda) \rbrack = \lbrack N_{k-1}:\nabla(\lambda) \rbrack + 1$, we have $\dim_K \Hom_\mc{C}(\Delta(\lambda),N) = \dim_K \Hom_\mc{C}(\Delta(\lambda),N_{k-1})+1$, so we need one more basis element. Applying $\Hom_{\mathcal{C}}(\Delta(\lambda),-)$ to the exact sequence \eqref{cellularity_hom_exact_sequence_2} yields an exact sequence
	\begin{equation}
		0 \rarr \Hom_{\mathcal{C}}(\Delta(\lambda),N_{k-1}) \rarr \Hom_{\mathcal{C}}(\Delta(\lambda),N) \overset{\pi \circ }{\rarr} \Hom_{\mathcal{C}}(\Delta(\lambda),\nabla(\lambda)) \rarr 0 \;,
	\end{equation}
	where the zero term on the right is due to \Autoref{ass_ext_vanishing} and the fact that $N$ has a costandard filtration. Hence, there is $g^\lambda:\Delta(\lambda) \rarr N$ such that $\pi \circ g^\lambda = c^\lambda$, where $c^\lambda$ is as in (\ref{c_lambda_def}). Then the set $\mrm{inc}(G_{N_{k-1}}^\lambda) \cup \lbrace g^\lambda \rbrace$ is a basis of $\Hom_\mc{C}(\Delta(\lambda),N)$. Let $\hat{g}^\lambda:T(\lambda) \rarr N$ be any lift and set $\hat{G}_N^\lambda \coloneqq \mrm{inc}(\hat{G}_{N_{k-1}}^\lambda) \cup \lbrace \hat{g}^\lambda \rbrace$.

	With these definitions of $\hat{G}_N^\mu$ it follows as in the proof of \cite[Proposition 3.3]{AST} that $\hat{G}_N\hat{F}_M$ is a basis of $\Hom_\mc{C}(M,N)$, so (\ref{cellular_newsection_ingredients:basis}) holds.

	We still have to argue that this basis satisfies (\ref{cellular_newsection_ingredients:lemma45}). Let $\phi \in \langle \hat{G}_N^\lambda \hat{F}_M^\lambda \rangle_K$ be non-zero. First, assume that we are in the case $\lambda \neq \mu$ of the recursive definition of $\hat{G}_N$ from above. Then, by definition of $\hat{G}_N^\lambda\hat{F}_M^\lambda$, there is a non-zero $\tilde{\phi} \in \langle \hat{G}_{N_{k-1}}^\lambda \hat{F}_M^\lambda \rangle_K$ with $\phi = \mrm{inc}(\tilde{\phi})$. By the induction hypothesis, we have $\tilde{\phi}_\lambda \neq 0$. Since $N_{k-1} \rarr N$ is an embedding, it follows that $\phi_\lambda \geq \tilde{\phi}_\lambda$, and $\phi_\lambda \neq 0$ as claimed. Now, assume that we are in the case $\lambda = \mu$. By definition of $\hat{G}_N^\lambda$ we can write $\phi = \mrm{inc}(\wt{\phi}) + \phi'$, where $\wt{\phi} \in \langle \hat{G}_{N_{k-1}}^\lambda F_M^\lambda \rangle_K$ and $\phi' \in \langle \hat{g}^\lambda F_M^\lambda \rangle_K$. First, assume that $\phi' \neq 0$. Because of the exact sequence (\ref{cellularity_hom_exact_sequence}) we have $\pi \circ \phi' \neq 0$. This is a morphism $M \rarr \nabla(\lambda)$, and since it is non-zero, we have $(\pi \circ \phi')_\lambda \neq 0$. Since $\pi \circ \mrm{inc}(\wt{\phi}) = 0$, we have $\pi \circ \phi = \pi \circ \phi'$. Thus $(\pi \circ \phi)_\lambda \neq 0$, implying that $\phi_\lambda \neq 0$ too. On the other hand, if $\phi' = 0$, we must have $\wt{\phi} \neq 0$ and the same argument, as above, shows that $\phi_\lambda \neq 0$. This concludes the proof of (\ref{cellular_newsection_ingredients:basis}) and (\ref{cellular_newsection_ingredients:lemma45}). \\

	Let us now choose bases and lifts satisfying (\ref{cellular_newsection_ingredients:basis}) and (\ref{cellular_newsection_ingredients:lemma45}). If $(c_{ij}^\mu)_\nu \neq 0$, then by definition $\lbrack \Image c_{ij}^\mu : L(\nu) \rbrack \neq 0$. This implies in particular that $\lbrack \Image \hat{g}_i^\mu : L(\nu) \rbrack \neq 0$. Recall that $\hat{g}_i^\mu$ is a morphism $T(\mu) \rarr N$. Hence, $\Image \hat{g}_i^\mu \simeq T(\mu)/\Ker \hat{g}_i^\mu$ and therefore $\nu \leq \mu$. In other words, $(c_{ij}^\mu)_\nu = 0$ unless $\nu \leq \mu$. Moreover, as $0 \neq c_{ij}^\mu \in \hat{G}_N^\mu \hat{F}_M^\mu$, we know from (\ref{cellular_newsection_ingredients:lemma45}) that $(c_{ij}^\mu)_\mu \neq 0$. Hence, $c_{ij}^\mu \in \Hom_\mc{C}(M,N)^{\leq \lambda}$ if and only if $\mu \leq \lambda$. In particular, $\langle (\hat{G}_N\hat{F}_M)^{\leq \lambda} \rangle_K \subs \Hom_\mc{C}(M,N)^{\leq \lambda}$. By assumption $(\hat{G}_N\hat{F}_M)^{\leq\lambda}$ is linearly independent, so we just need to show that it spans $\Hom_\mc{C}(M,N)^{\leq \lambda}$. Let $\phi \in \Hom_{\mc{C}}(M,N)^{\leq \lambda}$ be non-zero. Since $\hat{G}_N\hat{F}_M$ is a basis of $\Hom_\mc{C}(M,N)$, we can write
	\[
	\phi = \sum_{ \substack{\mu \in \Lambda \\ i \in \mc{I}^\mu , \ j \in \mc{J}^\mu }} a_{ij}^\mu c_{ij}^\mu
	\]
	with certain $a_{ij}^\mu \in K$, not all zero. Choose $\mu \in \Lambda$ maximal with the property that $a_{i j}^\mu \neq 0$ for some $i \in \mc{I}^\mu$ and $j \in \mc{J}^\mu$. Let
	\[
	\phi^\mu \coloneqq \sum_{i\in \mc{I}^\mu,j\in \mc{J}^\mu} a_{ij}^\mu c_{ij}^\mu \quad \tn{and} \quad \phi^{\neq \mu} \coloneqq \sum_{ \substack{\nu \neq \mu \\ i \in \mc{I}^\nu, j \in \mc{J}^\nu}} a_{ij}^\nu c_{ij}^\nu \;,
	\]
	so $\phi = \phi^\mu + \phi^{\neq \mu}$. Note that $\phi^\mu \in \langle \hat{G}_N^\mu \hat{F}_M^\mu \rangle_K$ is non-zero. Hence, we know from (\ref{cellular_newsection_ingredients:lemma45}) that $(\phi^\mu)_\mu \neq 0$. Moreover, the maximality of $\mu$ and the arguments above show that $(\phi^{\neq \mu})_\mu = 0$. Hence, $\phi_\mu = (\phi^\mu + \phi^{\neq \mu})_\mu = (\phi^\mu)_\mu \neq 0$ by \Autoref{f_plus_g_weight}. By definition of $\Hom_\mc{C}(M,N)^{\leq \lambda}$, this implies $\mu \leq \lambda$. Hence, by what we have said above, we have $c_{ij}^\mu \in \Hom_{\mc{C}}(M,N)^{\leq \lambda}$ for all $i,j$. In total, this shows that $\phi \in \langle (\hat{G}_N\hat{F}_M)^{\leq \lambda} \rangle_K$.

	The second statement in  (\ref{cellular_newsection_ingredients:filtered_basis}) follows analogously. \\

  To prove part (\ref{cellular_newsection_ingredients:disjoint}), let $\phi \in \hat{G}_N^\lambda\hat{F}_M^\lambda \cap \hat{G}_N^\mu\hat{F}_M^\mu$. Note that $\phi \neq 0$. Then we know from part (\ref{cellular_newsection_ingredients:lemma45}) that $\phi_\lambda \neq 0$ and from part (\ref{cellular_newsection_ingredients:filtered_basis}) we know that $\phi \in \Hom_{\mathcal{C}}(M,N)^{\leq \lambda}$. Analogously, we know that $\phi_\mu \neq 0$ and $\phi \in \Hom_{\mathcal{C}}(M,N)^{\leq \mu}$. Now, $\phi  \in  \Hom_{\mathcal{C}}(M,N)^{\leq \lambda}$ and $\phi_\mu \neq 0$ implies that $\mu \leq \lambda$. Analogously, we conclude that $\lambda \leq \mu$ and thus $\lambda = \mu$.
\end{proof}

\begin{proof}[Proof of \Autoref{cellular_newsection_basis_independent}]
	Having established \Autoref{cellular_newsection_ingredients}, the claim can now be proven by the exact same arguments as in the proofs of \cite[Lemma 3.6]{AST} and \cite[Lemma 3.7]{AST}. Namely, first fix $G_N \coloneqq \bigcup_{\lambda \in \Lambda} G_N^\lambda$ and a lift $\hat{G}_N \coloneqq \{ \hat{g}^\lambda \mid \lambda \in \Lambda, i \in \mathcal{I}^\lambda \}$ of $G_N$ satisfying the properties in \Autoref{cellular_newsection_ingredients}. Let $\tilde{G}_N \coloneqq \{ \tilde{g}_i^\lambda \mid \lambda \in \Lambda, i \in \mathcal{I}^\lambda \}$ be any other lift of $G_N$ (not necessarily satisfying the properties in \Autoref{cellular_newsection_ingredients}). By construction, we have $\hat{g}_i^\lambda \circ i^\lambda = g_i^\lambda = \tilde{g}_i^\lambda \circ i^\lambda$, hence $(\hat{g}_i^\lambda - \tilde{g}_i^\lambda) \circ i^\lambda = 0$. Therefore $\Delta(\lambda)= \Image i^\lambda \subseteq \Ker(\hat{g}_i^\lambda - \tilde{g}_i^\lambda)$. Since $\Image(\hat{g}_i^\lambda - \tilde{g}_i^\lambda) \simeq T(\lambda)/\Ker( \hat{g}_i^\lambda - \tilde{g}_i^\lambda)$, we have $( \hat{g}_i^\lambda - \tilde{g}_i^\lambda)_\lambda = 0$. Let $c_{ij}^\lambda \coloneqq \hat{g}_i^\lambda \circ \hat{f}_j^\lambda$ and set $\tilde{c}_{ij}^\lambda \coloneqq \tilde{g}_i^\lambda \circ \hat{f}_j^\lambda$. Then
	\begin{equation}
		c_{ij}^\lambda - \tilde{c}_{ij}^\lambda = (\hat{g}_i^\lambda - \tilde{g}_i^\lambda) \circ \hat{f}_j^\lambda
	\end{equation}
	and hence $(c_{ij}^\lambda - \tilde{c}_{ij}^\lambda)_\lambda = 0$ as well, i.e. $c_{ij}^\lambda - \tilde{c}_{ij}^\lambda \in \Hom_{\mathcal{C}}(M,N)^{< \lambda}$. It thus follows from part (\ref{cellular_newsection_ingredients:filtered_basis}) of \Autoref{cellular_newsection_ingredients} that the transformation matrix from  $\hat{G}_N \hat{F}_M = \{ c_{ij}^\lambda \mid \lambda \in \Lambda, i \in \mathcal{I}^\lambda, j \in \mathcal{J}^\lambda \}$ to $\tilde{G}_N \hat{F}_M = \{ \tilde{c}_{ij}^\lambda \mid \lambda \in \Lambda, i \in \mathcal{I}^\lambda, j \in \mathcal{J}^\lambda \}$ is unitriangular. Since $\hat{G}_N \hat{F}_M$ is a basis of $\Hom_{\mathcal{C}}(M,N)$ by \Autoref{cellular_newsection_ingredients}, we conclude that $\tilde{G}_N \hat{F}_M$ is a basis of $\Hom_{\mathcal{C}}(M,N)$ as well.

	A similar change-of-basis argument shows that we can replace the $G_N^\lambda$ by any other basis of $\Hom_{\mathcal{C}}(\Delta(\lambda), N)$ and still get a basis after choosing lifts; see \cite[Lemma 3.7]{AST}.
\end{proof}

\subsection{Standard bases}

Now, for a tilting object $T \in \mathcal{C}$ let us define
\begin{equation}
\mathcal{P}_T \coloneqq \lbrace \lambda \in \Lambda \mid G_T^\lambda F_T^\lambda \neq \emptyset \rbrace = \lbrace \lambda \in \Lambda \mid \lbrack T:\Delta(\lambda) \rbrack \neq 0 \tn{ and } \lbrack T:\nabla(\lambda) \rbrack \neq 0 \rbrace \;.
\end{equation}
We equip this set with the partial order $\leq$ from $\Lambda$ and set
\begin{equation}
E_T \coloneqq \End_{\mc{C}}(T) \;.
\end{equation}

The point of this paper is that the basis $\hat{G}_T \hat{F}_T$ of the algebra $E_T$ has special combinatorial properties, namely it is a \emph{standard basis} in the sense of Du and Rui \cite{DuRui}. We recall the precise definition.

\begin{defn}[Du–Rui] \label{standard_datum_def}
	A \word{standard basis} of a finite-dimensional $K$-algebra $E$ is a $K$-basis $\mathcal{B}$ of $E$ which is fibered over a poset $\Lambda$, i.e., $\mc{B} = \coprod_{\lambda \in \Lambda} \mc{B}^\lambda$, together with indexing sets $\mathcal{I}^{\lambda}$ and $\mathcal{J}^{\lambda}$ for any $\lambda \in \Lambda$ such that
	\begin{equation}
		 \mc{B}^\lambda = \{ c_{ij}^{\lambda} \  | \ (i,j) \in \mathcal{I}^{\lambda} \times \mathcal{J}^{\lambda} \} \;,
	\end{equation}
	and for any $\varphi \in E$ and $c_{ij}^\lambda \in \mc{B}^{\lambda}$ we have
		\begin{equation}
		\varphi \cdot c^{\lambda}_{ij} \equiv \sum_{k \in \mathcal{I}^{\lambda}} r_{k}^\lambda(\varphi,i) c_{k j}^{\lambda} \ \mod \ E^{< \lambda} \;,
	\end{equation}
		\begin{equation}
		c^{\lambda}_{ij} \cdot \varphi \equiv \sum_{l \in \mathcal{J}^{\lambda}} r_{l}^\lambda(j,\varphi) c_{il}^{\lambda} \ \mod \ E^{< \lambda} \;,
	\end{equation}
		where $r_{k}^\lambda(\varphi,i), r_{l}^\lambda(j,\varphi) \in K$ are independent of $j$ and $i$, respectively. Here, $E^{<\lambda}$ is the subspace of $E$ spanned by the set $\bigcup_{\mu < \lambda} \mathcal{B}^\mu$.
\end{defn}

Note that a standard basis is not just a basis $\mathcal{B}$ but comes with additional data, namely a poset $\Lambda$, the decomposition of $\mathcal{B}$ into the $\mathcal{B}^\lambda$, and a particular indexing scheme for the elements in $\mathcal{B}^\lambda$. It would be more precise to say ``standard datum'' instead of ``standard basis'' but we prefer the latter terminology.

A cellular basis (introduced earlier by Graham and Lehrer \cite{Graham-Lehrer-Cellular}) is a standard basis which behaves symmetrically under an anti-involution on the algebra.

\begin{defn}
A \word{cellular basis} (in the sense of Graham-Lehrer) of a finite-dimensional $K$-algebra $E$ is a standard basis $\mathcal{B}$ together with an algebra anti-involution $\iota$ on $E$ such that $\mathcal{I}^\lambda = \mathcal{J}^\lambda$ for all $\lambda \in \Lambda$ and
\begin{equation}
	\iota(c_{i,j}^\lambda) = c_{j,i}^\lambda \;,
\end{equation}
for all $(i,j) \in \mathcal{I}^\lambda \times \mathcal{J}^\lambda$.
\end{defn}

We will address involutions and cellularity in \Autoref{subsec_duality} and first consider the more general concept of a standard basis. With Theorem \ref{cellular_newsection_basis_independent} established, the exact same arguments as in the proof of \cite[Theorem 3.9]{AST} now show the first part of our main theorem from the introduction:

\begin{thm} \label{end_tilting_is_cellular}
	For any tilting object $T \in \mathcal{C}$ the basis $\hat{G}_T\hat{F}_T = \coprod_{\lambda \in \Lambda} \hat{G}_T^\lambda \hat{F}_T^\lambda$ is a standard basis of the $K$-algebra $E_T$.
\end{thm}

\begin{proof}
By construction, the elements in $\hat{G}_T^\lambda \hat{F}_T^\lambda$ are of the form $c_{ij}^\lambda = \hat{g}_i^\lambda \circ \hat{f}_j^\lambda$ with $i \in \mathcal{I}_T^\lambda$ and $j \in \mathcal{J}_T^\lambda$. Let $\varphi \in E_T$. Since $G_T^\lambda$ is a basis of $\Hom_{\mathcal{C}}(\Delta(\lambda), T)$, we can write
\begin{equation} \label{end_tilting_is_cellular_equ1}
	\varphi \circ g_i^\lambda = \sum_{k \in \mathcal{I}_T^\lambda } r_{k}^\lambda(\varphi,i) g_k^\lambda
\end{equation}
with suitable $r_{k}^\lambda(\varphi,i) \in K$. Note that $\Hom_{\mathcal{C}}(\Delta(\lambda),T) = \Hom_{\mathcal{C}}(\Delta(\lambda),T)^{\leq \lambda}$ because $\Delta(\lambda)$ is of highest weight $\lambda$. For any $k \in \mathcal{I}^\lambda$ we have $\hat{g}_k^\lambda \circ i^\lambda = g_k^\lambda$ by construction. Hence, from \eqref{end_tilting_is_cellular_equ1} we get
\begin{equation}
\varphi \circ \hat{g}_i^\lambda \circ i^\lambda = \sum_{k \in \mathcal{I}_T^\lambda} r_k^\lambda(\varphi,i) \hat{g}_k^\lambda \circ i^\lambda \;,
\end{equation}
i.e.
\begin{equation}
\left( \varphi \circ \hat{g}_i^\lambda  - \sum_{k \in \mathcal{I}_T^\lambda} r_k^\lambda(\varphi,i) \hat{g}_k^\lambda \right) \circ i^\lambda = 0 \;,
\end{equation}
The morphism in parentheses is contained in $\Hom_{\mathcal{C}}(T(\lambda),T)^{\leq \lambda}$ because $T(\lambda)$ is of highest weight $\lambda$, and the above equation now implies that
\begin{equation}
	\varphi \circ \hat{g}_i^\lambda  - \sum_{k \in \mathcal{I}_T^\lambda} r_k^\lambda(\varphi,i) \hat{g}_k^\lambda \in \Hom_{\mathcal{C}}(T(\lambda),T)^{< \lambda} \;.
\end{equation}
We can thus conclude that
\begin{equation}
\varphi \circ \hat{g}_i^\lambda \circ \hat{f}_j^\lambda - \sum_{k \in \mathcal{I}_T^\lambda} r_k^\lambda(\varphi,i) \hat{g}_k^\lambda \circ \hat{f}_j^\lambda \in \Hom_{\mathcal{C}}(T,T)^{<\lambda} = E_T^{< \lambda} \;,
\end{equation}
i.e.
\begin{equation}
\varphi \circ c_{ij}^\lambda \equiv \sum_{k \in \mathcal{I}_T^\lambda} r_k^\lambda(\varphi,i) c_{kj}^\lambda \bmod  E_T^{< \lambda} \;.
\end{equation}
Similarly, one proves the relation for the multiplication $c_{ij}^\lambda \circ \varphi$.
\end{proof}

\subsection{Remark on cell modules} \label{rem_on_cell_modules}

Generalizing the cell modules defined by Graham–Lehrer \cite[Definition 2.1]{Graham-Lehrer-Cellular}, Du–Rui \cite[Definition 2.1.2]{DuRui} defined (cellular) (co-)standard modules for each $\lambda \in \mathcal{P}_T$. The (cellular) \word{standard module} $\Delta_T(\lambda)$ attached to $\lambda \in \mc{P}_T$ is a left $E_T$-module defined using the coefficients for left multiplication appearing in the definition of a standard basis. Similarly, the (cellular) \word{costandard module} $\nabla_T(\lambda)$ is a left $E_T$-module which is the dual of an analogous right module defined using right multiplication. Though related, these should not be confused with the standard and costandard modules in $\mc{C}$ given in \Autoref{ass_ext_vanishing}. In the terminology of Graham–Lehrer the left $E_T$-module $\Delta_T(\lambda)$ is the left cell module attached to $\lambda$ and the dual of $\nabla_T(\lambda)$, a right $E_T$-module, is the right cell module (or dual cell module) attached to $\lambda$.

The arguments in \cite[\S4]{AST} show that we have
\begin{equation}\label{eq:standardcostandardiso}
\Delta_T(\lambda) \simeq \Hom_{\mathcal{C}}(\Delta(\lambda),T) \quad \tn{and} \quad \nabla_T(\lambda)^* \simeq \Hom_{\mathcal{C}}(T,\nabla(\lambda)) \;.
\end{equation}
This is due to the particular form of the standard basis and shows that these modules do not depend on the choice of the bases $F^\lambda$ and $G^\lambda$. Generalizing the construction of Graham–Lehrer, Du–Rui \cite[2.3]{DuRui} defined for any $\lambda \in \mathcal{P}_T$ a bilinear pairing $\beta_\lambda$ between $\Delta_T(\lambda)$ and $\nabla_T(\lambda)^*$. It follows from \cite[Theorem 2.4.1]{DuRui} that the subset
\begin{equation}
\ul{\mathcal{P}}_{T} \coloneqq \lbrace \lambda \in \mathcal{P}_T \mid \beta_\lambda \neq 0 \rbrace \subs \mathcal{P}_T
\end{equation}
classifies the simple $E_T$-modules: the cellular standard module $\Delta_T(\lambda)$ has simple head $L_T(\lambda)$ if and only if $\beta_\lambda \neq 0$, and $\lambda \mapsto L_T(\lambda)$ is a bijection between $\ul{\mathcal{P}}_{T}$ and the set of isomorphism classes of simple left $E_T$-modules.

As shown in \cite[Definition 1.2.1]{DuRui}, the opposite algebra $E_T^\op$ is naturally equipped with a standard datum, the \word{opposite standard datum},  which has the same indexing poset $\mc{P}_T$ and flipped basis parts. The standard and costandard modules for this standard datum are then given by
\begin{equation} \label{opposite_cellular_standard_modules}
\Delta_T^\op(\lambda) = \nabla_T(\lambda)^* = \Hom_{\mc{C}}(T,\nabla(\lambda)) \;\ \tn{and} \;\ \nabla_T^\op(\lambda) = \Delta_T(\lambda)^* = \Hom_{\mc{C}}(\Delta(\lambda),T)^* \;,
\end{equation}
respectively.
The arguments by Andersen–Stroppel–Tubbenhauer \cite[\S4, Theorem 4.12, Theorem 4.13]{AST} can be used word-for-word to prove the following two theorems:

\begin{thm} \label{cell_newsection_dim_simples}
	If $\lambda \in \ul{\mathcal{P}}_{T}$, then $\dim L_T(\lambda)$ is equal to the multiplicity of $T(\lambda)$ as a direct summand of $T$.
\end{thm}

\begin{thm} \label{cell_newsection_semisimple}
	The algebra $E_T$ is semisimple if and only if $T \in \mathcal{C}$ is semisimple.
\end{thm}

\subsection{Standard dualities and cellularity} \label{subsec_duality}
We will now turn to the question of the existence of an anti-involution on $E_T$ making the standard basis a cellular basis. In light of our categorical approach, the anti-involution should come from a duality on $\mathcal{C}$. However, this requires some assumptions on the duality, leading us to the notion of a ``standard duality''. We will keep this section abstract and turn to explicit examples in \Autoref{sec_module_dualities}.

By a \word{duality} we mean a contravariant $K$-linear functor $\bbD \colon \mathcal{C} \to \mathcal{C}$ such that there is an isomorphism
\begin{equation}
	\xi \colon \id_{\mathcal{C}} \to \bbD^2 = \bbD \circ \bbD
\end{equation}
of $K$-linear functors $\mathcal{C} \to \mathcal{C}$ satisfying
\begin{equation} \label{duality_relation}
\id_{\bbD(X)} = \bbD(\xi_X) \circ \xi_{\bbD(X)} \;.
\end{equation}
It is an elementary categorical fact that a duality is an equivalence, see e.g. \cite[Lemma 2.3]{Shimizu-FS}.
We say that an object $T \in \mathcal{C}$ is \word{self-dual} if $\bbD(T)$ is isomorphic to $T$. Suppose that $T$ is self-dual and choose an isomorphism
\begin{equation}
	\Phi \coloneqq \Phi_T \colon \bbD(T) \to T \;.
\end{equation}
As in \cite[Theorem 1.2.1]{CPS-Stratifying} we define a $K$-algebra anti-morphism
\begin{equation} \label{alpha_definition}
	\alpha^{-1} \coloneqq \alpha_{\Phi}^{-1} \colon E_T \to E_T
\end{equation}
on the endomorphism algebra $E_T$ of $T$ via
\begin{equation}
\alpha^{-1}(\varphi) \coloneqq \Phi \circ \bbD(\varphi) \circ \Phi^{-1}
\end{equation}
for $\varphi \in E_T$. We compute the square $\alpha^{-2}$ to be
\begin{equation}
\alpha^{-2}(\varphi) = \Phi \circ \bbD(\Phi^{-1}) \circ \bbD^2(\varphi) \circ \bbD(\Phi) \circ \Phi^{-1} \;.
\end{equation}
By naturality of $\xi$, the diagram
\begin{equation}
	\begin{tikzcd}
		T \ar[r,"\xi_T"] \ar[d,swap,"\varphi"] & \bbD^2(T) \ar[d,"\bbD^2(\varphi)"] \\
		T \ar[r, swap,"\xi_T"] & \bbD^2(T)
	\end{tikzcd}
\end{equation}
commutes, i.e.
\begin{equation}
	\bbD^2(\varphi) = \xi_T \circ \varphi \circ \xi_T^{-1} \;.
\end{equation}
Hence, defining
\begin{equation}
	a \coloneqq a_{\Phi,\xi} \coloneqq \Phi \circ \bbD(\Phi^{-1}) \circ \xi_T \in E_T^\times
\end{equation}
we get
\begin{equation}
\alpha^{-2}(\varphi) = a \circ \varphi \circ a^{-1} \;.
\end{equation}
This means that $\alpha^{-2}$ is an inner automorphism of $E_T$, given by conjugation by $a$, and hence $\alpha^{-1}$ is an anti-automorphism. In particular, $E_T \simeq E_T^{\mathrm{op}}$. We obtain an algebra anti-automorphism
\begin{equation}
	\alpha = \alpha_\Phi \colon E_T \to E_T
\end{equation}
whose square is conjugation by $a_{\Phi,\xi}^{-1}$. There is no reason why $\alpha^2$ should be trivial, equivalently, why $\alpha$ should be an anti-\emph{involution}. This motivates the following definition.

\begin{defn} \label{defn_duality_fixed}
We say that $(T,\Phi_T)$ is a \word{fixed point} of $\bbD$ if $\alpha=\alpha_{\Phi_T}$ is an anti-involution, i.e.
\begin{equation}
	\Phi_T \circ \bbD(\Phi_T^{-1}) \circ \xi_T = \id_T \;.
\end{equation}
We say that $T \in \mathcal{C}$ is \word{fixed} by $\bbD$ if $\Phi_T$ can be chosen such that $(T,\Phi_T)$ is a fixed point.
\end{defn}

It is clear that if $(T,\Phi_T)$ and $(T',\Phi_{T'})$ are fixed points of $\bbD$, then so is their direct sum $(T \oplus T', \Phi_{T} \oplus \Phi_{T'})$.%\todo{This is correct, right?}

\begin{rem}
Our definition of a duality goes back to \cite{QSS-QuadraticForms}. What we call a ``fixed point'' is also called a ``symmetric space'' in, e.g., \cite{Balmer-Witt}. The reason for this terminology will become clear in \Autoref{sec_module_dualities}.
\end{rem}

Now, suppose that $\bbD$ \emph{exchanges} standard and costandard objects, i.e.
\begin{equation}
	\bbD(\nabla(\lambda)) \simeq \Delta(\lambda)
\end{equation}
for all $\lambda \in \Lambda$.

\begin{lem}
Every simple object $L(\lambda)$ is fixed by $\bbD$ and all tilting modules are self-dual.
\end{lem}

\begin{proof}
Since a duality exchanges heads and socles of objects, it is clear that all simple objects $L(\lambda)$ are self-dual. Since the endomorphism algebra of $L(\lambda)$ is a (commutative) field, this implies that $L(\lambda)$ is fixed by $\bbD$.

It is clear that $\bbD$ maps $\mathcal{C}^\nabla$ into $\mathcal{C}^\Delta$ and vice versa, hence $\bbD$ maps tilting objects to tilting objects. Since $\bbD$ preserves indecomposability of objects, it follows that $\bbD(T(\lambda))$ is isomorphic to $T(\mu)$ for some $\mu$. But, as $\bbD$ fixes the simple objects and $\lbrack T(\lambda) \colon L(\lambda) \rbrack = 1$ and $\lbrack T(\lambda) \colon L(\nu) \rbrack \neq 0$ implies $\nu \leq \lambda$, the same property holds for $T(\mu)$. Thus, we have $\lambda = \mu$ and we conclude that
\begin{equation} \label{tilting_self_dual}
	\bbD(T(\lambda)) \simeq T(\lambda) \;,
\end{equation}
for all $\lambda \in \Lambda$. This shows that all indecomposable tilting objects (and thus all tilting objects) are self-dual.
\end{proof}

This leads us to the following definition.

\begin{defn}
	A \word{standard duality} on a standard category $\mathcal{C}$ is a duality $\bbD$ on~$\mathcal{C}$ that exchanges standard and costandard objects, and for which every indecomposable tilting object is a fixed point (as in \Autoref{defn_duality_fixed}).
\end{defn}

Note that every tilting module is a fixed point under a standard duality.

\begin{rem}
Dualities on highest weight categories are studied in, e.g., \cite{CPS-Duality-in-highest-weight-89, CPS-Stratifying}. In loc. cit. the relation \eqref{duality_relation} for a duality is not required but we feel it is natural. A duality fixing the simple objects is in loc. cit. called a ``strong duality''. In \cite{BrundanStroppel}, the authors introduce the notation of ``Chevalley duality''. This is expected to be closely related to standard duality, at least for finite highest weight categories.
\end{rem}

We assume from now on that $\bbD$ is a standard duality. Let $T \in \mathcal{C}$ be a tilting object. Since $T$ is fixed by $\bbD$, we can choose an isomorphism
\begin{equation}
	\Phi_T \colon \bbD(T) \to T
\end{equation}
such that
\begin{equation} \label{T_fixed_point_equ}
\Phi_T \circ \bbD(\Phi_T^{-1}) \circ \xi_T = \id_T \;.
\end{equation}
In particular, $\bbD$ induces an involution on $E_T \coloneqq \End_{\mathcal{C}}(T)$.

Our goal is to prove that this involution will make specific choices of standard bases from Theorem \ref{end_tilting_is_cellular} into cellular bases. These specific choices are obtained by choosing $G_T$ and $\hat{G}_T$, and then using $\bbD$ to obtain $\hat{F}_T^\lambda$, i.e. $\hat{F}_T^\lambda$ is in a sense the ``dual'' of $\hat{G}_T^\lambda$. To make this precise, let us choose for each $\lambda \in \Lambda$ an isomorphism
\begin{equation}
	\Phi_{T(\lambda)} \colon \bbD( T(\lambda) ) \to T(\lambda)
\end{equation}
such that
\begin{equation} \label{T_lambda_fixed_point_equ}
\Phi_{T(\lambda)} \circ \bbD(\Phi_{T(\lambda)}^{-1}) \circ \xi_{T(\lambda)} = \id_{T(\lambda)} \;.
\end{equation}
Recall that the projection $\pi^\lambda \colon T(\lambda) \twoheadrightarrow \nabla(\lambda)$ is unique up to scalars. Hence, there is a unique subobject $Q(\lambda)$ of $T(\lambda)$ such that the quotient $T(\lambda)/Q(\lambda)$ is isomorphic to $\nabla(\lambda)$. We have an epimorphism
\begin{equation}
	\bbD(T(\lambda)) \overset{\Phi_{T(\lambda)}}{\rightarrow} T(\lambda) \overset{\pi^\lambda}{\rightarrow} \nabla(\lambda)
\end{equation}
with kernel $\Phi_{T(\lambda)}^{-1}(Q(\lambda))$. Moreover, we have an epimorphism $\bbD(i^\lambda) \colon \bbD(T(\lambda)) \to \bbD(\Delta(\lambda)) \simeq \nabla(\lambda)$, the kernel of which must also equal $\Phi_{T(\lambda)}^{-1}(Q(\lambda))$. It follows that $\Phi_{T(\lambda)}$ induces an isomorphism
\begin{equation}
	\overline{\Phi}_{T(\lambda)} \colon \bbD(\Delta(\lambda)) \to \nabla(\lambda)
\end{equation}
making the diagram
\begin{equation} \label{pi_dual_commutativity}
	\begin{tikzcd}
		\bbD(T(\lambda)) \ar[r, "\Phi_{T(\lambda)}"] \ar[d, swap, "\bbD(i^\lambda)"] & T(\lambda) \ar[d, "\pi^\lambda"]\\
		\bbD(\Delta(\lambda)) \ar[r, swap, "\overline{\Phi}_{T(\lambda)}"] & \nabla(\lambda)
	\end{tikzcd}
\end{equation}
commute.

Let us now choose for all $\lambda \in \Lambda$ a basis $G_T^\lambda = \{ g_i^\lambda \mid i \in \mathcal{I}_T^\lambda \}$ of $\Hom_{\mathcal{C}}(\Delta(\lambda),T)$ and a lift $\hat{G}_T^\lambda = \{ \hat{g}_i^\lambda \mid i \in \mathcal{I}_T^\lambda \}$ of $G_T^\lambda$ as in \Autoref{sec_basic_construction}. For $g_i \in G_T^\lambda$ we define
\begin{equation}
	f_i^{\lambda} \coloneqq \overline{\Phi}_{T(\lambda)} \circ \bbD(g_i^\lambda) \circ \Phi_T^{-1} \colon T \to \nabla(\lambda)
\end{equation}
and for $\hat{g}_i^\lambda \in \hat{G}_T^\lambda$ we define
\begin{equation}
\hat{f}_i^\lambda \coloneqq \Phi_{T(\lambda)} \circ \bbD(\hat{g}_i^\lambda) \circ \Phi_T^{-1} \colon T \to T(\lambda) \;.
\end{equation}
Since $\bbD$ is a duality, it is clear that
\begin{equation}
	\bbD(G_T^{\lambda}) \coloneqq \{ f_i^\lambda \mid i \in \mathcal{I}_T^\lambda \}
\end{equation}
is a basis of $\Hom_{\mathcal{C}}(T,\nabla(\lambda))$. It follows from the commutative diagram \eqref{pi_dual_commutativity} that
\begin{equation}
	\bbD(\hat{G}_T^\lambda) \coloneqq \{ \hat{f}_i^\lambda \mid i \in \mathcal{I}_T^\lambda \}
\end{equation}
is a lift of $\bbD(G_T^\lambda)$. Hence, setting
\begin{equation}
	\bbD(\hat{G}_T) \coloneqq \coprod_{\lambda \in \Lambda} \bbD(\hat{G}_T^\lambda)
\end{equation}
we conclude from \Autoref{cellular_newsection_basis_independent} that $\hat{G}_T \bbD(\hat{G}_T)$ is a basis of $E_T$. We can now come to the second part of our main theorem from the introduction.

\begin{thm} \label{cellularity_thm_precise}
Let $\bbD$ be a standard duality on $\mathcal{C}$ and let $T \in \mathcal{C}$ be a tilting object. Then for any choice of $G_T$ and lift $\hat{G}_T$ the standard basis $\hat{G}_T \bbD(\hat{G}_T)$ of $E_T$ is a cellular basis with respect to the involution induced by $\bbD$.
\end{thm}

\begin{proof}
As usual, let $c_{ij}^\lambda \coloneqq \hat{g}_i^\lambda \hat{f}_j^\lambda$. All that remains to be proven is that
\begin{equation}
	\alpha(c_{ij}^\lambda) = c_{ji}^\lambda \;,
\end{equation}
where $\alpha$ is the involution on $E_T$ induced by $\bbD$ from \eqref{alpha_definition}. The right-hand side is equal to
\begin{equation}
c_{ji}^\lambda = \hat{g}_j^\lambda \circ \hat{f}_i^\lambda = \hat{g}_j^\lambda \circ \Phi_{T(\lambda)} \circ \bbD(\hat{g}_i^\lambda) \circ \Phi_T^{-1} \;.
\end{equation}
The left-hand side is equal to
\begin{equation} \label{cellularity_thm_precise_eq1}
  \begin{aligned}
\alpha(c_{ij}^\lambda) &= \Phi_T \circ \bbD(c_{ij}^\lambda) \circ \Phi_T^{-1}  = \Phi_T \circ \bbD\left( \hat{g}_i^\lambda \circ \Phi_{T(\lambda)} \circ \bbD(\hat{g}_j^\lambda) \circ \Phi_T^{-1} \right) \circ \Phi_T^{-1} \\
&= \Phi_T \circ \bbD(\Phi_T^{-1}) \circ \bbD^2(\hat{g}_j^\lambda) \circ \bbD(\Phi_{T(\lambda)}) \circ \bbD(\hat{g}_i^\lambda) \circ \Phi_T^{-1} \;.
\end{aligned}
\end{equation}
From the naturality of $\xi \colon \id_{\mathcal{C}} \to \bbD^2$ applied to $\hat{g}_j^\lambda$ we obtain a commutative diagram
\begin{equation}
\begin{tikzcd}
T(\lambda) \ar[r, "\xi_{T(\lambda)}"] \ar[d, swap,"\hat{g}_j^\lambda"]& \bbD^2(T(\lambda)) \ar[d, "\bbD^2(\hat{g}_j^\lambda)"] \\
T \ar[r, swap, "\xi_T"] & \bbD^2(T)
\end{tikzcd}
\end{equation}
i.e.
\begin{equation}
\bbD^2(\hat{g}_j^\lambda) = \xi_T \circ \hat{g}_j^\lambda \circ \xi_{T(\lambda)}^{-1} \;.
\end{equation}
Hence, from \eqref{cellularity_thm_precise_eq1} we get
\begin{align}
\alpha(c_{ij}^\lambda) & = \underbrace{\Phi_T \circ \bbD(\Phi_T^{-1}) \circ \xi_T}_{=\id_{T}} \circ \hat{g}_j^\lambda \circ \underbrace{\xi_{T(\lambda)}^{-1} \circ \bbD(\Phi_{T(\lambda)})}_{=\Phi_{T(\lambda)}} \circ \bbD(\hat{g}_i^\lambda) \circ \Phi_T^{-1} = c_{ji}^\lambda \;,
\end{align}
where the equalities under the braces come from the assumption that $(T,\Phi_T)$ and $(T(\lambda), \Phi_{T(\lambda)})$ are fixed points of $\bbD$, see \eqref{T_lambda_fixed_point_equ} and \eqref{T_fixed_point_equ}.
\end{proof}

\begin{rem}
\Autoref{cellularity_thm_precise} implies that the full subcategory $\mathcal{C}^t$ of tilting objects in $\mathcal{C}$ is a (strictly) object-adapted cellular category in the sense of Elias--Lauda \cite{EliasLauda}, see also \cite[\S11]{EMTW}.
\end{rem}

\subsection{Module dualities and examples of standard dualities} \label{sec_module_dualities}
On module categories there are natural dualities that we will call ``module dualities'' in this paper. For such dualities, being a self-dual object is related to being equipped with a non-degenerate bilinear form while being a fixed point is related to being equipped with a \emph{symmetric} non-degenerate bilinear form. This connection helps in verifying that a self-dual module is a fixed point—and showing that a duality is a standard duality—in examples.

Let $\cat{Mod}{K}$ be the category of $K$-vector spaces and consider the contravariant functor
\begin{equation} \label{canonical_duality}
	(-)^* \coloneqq \Hom_K(-,K) \colon \cat{Mod}{K} \to \cat{Mod}{K} \;.
\end{equation}
There is a natural monomorphism
\begin{equation} \label{xi_def}
\begin{array}{rcl}
	\xi_X \colon X & \longrightarrow & X^{**} \\
	x & \longmapsto & \mathrm{ev}_x, \ \mathrm{ev}_x(f) = f(x) \;.
\end{array}
\end{equation}
On the category $\cat{mod}{K}$ of finite-dimensional vector spaces, $\xi$ is an isomorphism and $(-)^*$ is a duality.

Let $A$ be a $K$-algebra and let $\cat{Mod}{A}$, respectively $\cat{Mod}{}{A}$, be the category of left, respectively right, $A$-modules. If $X \in \cat{Mod}{A}$, then $X^* = \Hom_K(X,K)$ is naturally a \emph{right} $A$-module with action $(fa)(x) \coloneqq f(ax)$ for $a \in A$, $f \in X^*$, and $x \in X$. The functor \eqref{canonical_duality} thus yields a contravariant functor
\begin{equation}
	(-)^* \colon \cat{Mod}{A} \to \cat{Mod}{}{A} \;.
\end{equation}

To transform right modules back to left modules we consider an algebra anti-involution $\tau \colon A \to A$. If $X \in \cat{Mod}{}{A}$, then let $X^\tau$ be the \emph{left} $A$-module with the same underlying vector space as $X$ but with $A$-action $ax \coloneqq x \tau(a)$. This defines a \emph{covariant} equivalence
\begin{equation}
	(-)^\tau \colon \cat{Mod}{}{A} \to \cat{Mod}{A} \;.
\end{equation}
Combined, we thus obtain a contravariant functor
\begin{equation}
	(-)^\tau \circ (-)^* \colon \cat{Mod}{A} \to \cat{Mod}{A} \;.
\end{equation}

For the rest of this section we assume that $\mathcal{C}$ is a $K$-linear subcategory of $\cat{Mod}{A}$.

\begin{defn}
A \word{module duality} on $\mathcal{C}$ is a duality of the form
\begin{equation}
\bbD_{\tau,\vee} \coloneqq (-)^\tau \circ (-)^\vee \colon \mathcal{C} \to \mathcal{C}
\end{equation}
for a $K$-linear subfunctor $(-)^\vee$ of $(-)^* \colon \mathcal{C} \to \cat{Mod}{}{A}$ such that $(X^\vee)^\tau \in \mathcal{C}$ for all $X \in \mathcal{C}$, and the duality isomorphism is the morphism $\xi$ of \eqref{xi_def}.
\end{defn}

\begin{ex}
Let $\mathcal{C} = \cat{mod}{A}$ be the category of finite-dimensional $A$-modules and let $(-)^\vee = (-)^*$. Then $\bbD_{\tau,\vee}$ is a module duality on $\cat{mod}{A}$.
\end{ex}

As we shall see below, there are natural infinite-dimensional examples as well which necessitates our more general definition.

\begin{lem} \label{self_dual_form}
If $X \in \mathcal{C}$ is self-dual under $\bbD_{\tau,\vee}$, then $X$ carries a non-degenerate bilinear form $\langle -, - \rangle$ which is moreover \emph{associative}, i.e. $\langle a x, y \rangle = \langle x, \tau(a)y \rangle$
for $x,y \in X$ and $a \in A$.
\end{lem}

\begin{proof}
Since $X$ is self-dual, there is an isomorphism $\Psi_X \colon X \to \bbD_{\tau,\vee}(X)$ in $\mathcal{C}$, and this induces a bilinear form on $X$ via
\begin{equation}
	\langle x,y \rangle \coloneqq \Psi_X(x)(y) \;.
\end{equation}
This form is non-degenerate since $\langle x,y \rangle = 0$ for all $y \in X$ means that $\Psi_X(x) = 0$ and hence $x =0$ because $\Psi_X$ is an isomorphism. The associativity follows from the fact that $\Psi_X$ is a morphism of $A$-modules:
\[
\langle ax,y \rangle = \Psi_X(ax)(y) = (\Psi_X(x) \tau(a) )(y) = \Psi_X(x)(\tau(a)y) = \langle x,\tau(a)y \rangle \;. \qedhere
\]
\end{proof}

\begin{lem} \label{fixed_symm_form}
	Let $X \in \mathcal{C}$.
	\begin{enumerate}
		\item If $X$ is a fixed point of $\bbD_{\tau,\vee}$, then $X$ carries a symmetric associative non-degenerate bilinear form.
		\item Conversely, if $X$ is equipped with a symmetric associative non-degenerate bilinear form such that the corresponding $A$-module morphism $\Psi_X \colon X \to (X^*)^\tau$ maps into $\bbD_{\tau,\vee}(X)$ and is an isomorphism in $\mathcal{C}$, then $X$ is a fixed point of $\bbD_{\tau,\vee}$.
	\end{enumerate}
\end{lem}

\begin{proof}
	Let $\bbD \coloneqq \bbD_{\tau,\vee}$. Then $X$ being a fixed point of $\bbD$ means there is an isomorphism $\Phi_X \colon \bbD(X) \to X$ such that $\Phi_X \circ \bbD(\Phi_X^{-1}) \circ \xi_X = \id_X$. Setting $\Psi_X \coloneqq \Phi_X^{-1} \colon X \to \bbD(X)$, we have $\Psi_X^{-1} \circ \bbD(\Psi_X) \circ \xi_X = \id_X$, i.e. $\bbD(\Psi_X) \circ \xi_X = \Psi_X$. By definition of $\xi$ in \eqref{xi_def} this means $\bbD(\Psi_X)(\mathrm{ev}_x) = \Psi_X(x)$ for all $x \in X$. Since $(-)^\vee$ is a subfunctor of $(-)^*$, we have $\bbD(\Psi_X)(\mathrm{ev}_x) = \mathrm{ev}_x \circ \Psi_X$ and therefore
	\begin{equation}
		\Psi_X(y)(x) = \mathrm{ev}_x \circ \Psi_X(y) = \Psi_X(x)(y) \;,
	\end{equation}
  where the first equality is simply by definition of $\mathrm{ev}_x$. This implies the two statements in the lemma.
\end{proof}

\begin{rem} \label{fixed_point_strategy}
In the examples below, we will use the following general strategy for showing that a self-dual object $X \in \mathcal{C}$ is actually a fixed point of $\bbD \coloneqq \bbD_{\tau,\vee}$. This strategy is motivated by the proof of \cite[Lemma 6.6]{Fiebig}. Since $X$ is self-dual, there is an isomorphism $\Psi_X \colon X \to \bbD(X)$ in $\mathcal{C}$. Let $\langle -, - \rangle$ be the corresponding associative non-degenerate bilinear form on $X$. We define the \word{symmetrization} of $\langle -,- \rangle$ to be the bilinear form $\langle -, - \rangle'$ on $X$ given by
\begin{equation}\label{eq:SYMofform}
	\langle x,y \rangle' \coloneqq \langle x,y \rangle + \langle y,x \rangle \;.
\end{equation}
This is an associative \emph{symmetric} bilinear form.
Due to the associativity, the corresponding map $\Psi_X' \colon X \to (X^*)^\tau$, $x \mapsto \langle x,- \rangle'$, is a morphism of $A$-modules. We will now make the following assumptions:
\begin{enumerate}
\item $\langle x, - \rangle' \in X^\vee$ for all $x \in X$ and $\Psi_X' \colon X \to \bbD(X)$ is a morphism in $\mathcal{C}$;
\item $X \in \mathcal{C}$ is indecomposable and of finite length;
%\item the characteristic of $K$ is not equal to 2.
\item the endomorphism $\Psi_X^{-1} \circ \Psi_X'$ of $X$ is not nilpotent.
\end{enumerate}
These assumptions, together with the Fitting lemma \cite[Lemma 5.3]{Krause-KS}, imply that $\Psi_X^{-1} \circ \Psi_X'$, and thus $\Psi_X'$, is an isomorphism in $\mathcal{C}$. Hence, $X$ is a fixed point of $\bbD$ by \Autoref{fixed_symm_form}. We note that, because $\langle x,x \rangle' = 2 \langle x,x \rangle$, it is necessary that the characteristic of $K$ is not equal to 2.
\end{rem}

\begin{ex} \label{duality_uqmod}
Consider the category $\cat{mod}{U_q}$ from \Autoref{uq_mod_examples}. Lusztig's integral form of the quantum group is a Hopf subalgebra of the generic quantum group, hence $U_q$ inherits a Hopf algebra structure as well, see \cite[Proposition 4.8]{Lusztig}. The antipode $S$ on $U_q$ is an anti-involution acting on the standard generators by
\begin{equation}
S(E_i) = -K_i^{-1} E_i \;, \quad S(F_i) = -F_i K_i \;, \quad S(K_i) = K_i^{-1} \;.
\end{equation}
The usual involution $\omega$, as in \cite[\S4.6]{Jantzen}, on the generic quantum group induces an involution on $U_q$, acting on the standard generators by
\begin{equation}
	\omega(E_i) = F_i \;, \quad \omega(F_i) = E_i \;, \quad \omega(K_i) = K_i \;.
\end{equation}
We thus obtain an anti-involution
\begin{equation}
	\tau \coloneqq \omega \circ S
\end{equation}
on $U_q$ with
\begin{equation}
\tau(E_i) = -K_i F_i \;, \quad \tau(F_i) = - E_i K_i^{-1} \;, \quad \tau(K_i) = K_i \;.
\end{equation}
Let $\bbD \coloneqq \bbD_{\tau,*}$ be the corresponding module duality. It preserves type-1 modules, hence is a duality on $\cat{mod}{U_q}$. It is shown in \cite{AST,AST-Notes} that~$\bbD$ exchanges standard and costandard modules. Assuming that the characteristic of the field $K$ is not equal to~2, we claim that $\bbD$ fixes the indecomposable tilting modules $T(\lambda)$ and is thus a standard duality on $\cat{mod}{U_q}$.

We will employ the strategy from \Autoref{fixed_point_strategy}. The first two assumptions in \Autoref{fixed_point_strategy} are clear, so all that remains to prove is that $\Psi^{-1} \circ \Psi'$ is not nilpotent. This follows from the fact that $U_q$-modules have weight spaces and that the highest weight space $T(\lambda)_\lambda$ of $T(\lambda)$ is 1-dimensional, see \cite{AST}. Namely, since $\Psi$ and $\Psi'$ are morphisms in
$\cat{mod}{U_q}$, we can consider their restrictions $\Psi_\lambda$ and $\Psi_\lambda'$ to the respective $\lambda$-weight space (here, we do not mean our new categorical version of restriction from \eqref{cellular_newsection_key_modification} but ``classical'' restriction to weight spaces). Let $0 \neq z \in T(\lambda)_\lambda$. Since $\Psi$ is an isomorphism, so is $\Psi_\lambda$ and therefore $\langle z,z \rangle \neq 0$. Hence, $\langle z,z \rangle' = 2 \langle z,z \rangle \neq 0$. This shows that $\Psi'_\lambda(z) \neq 0$. Hence, $(\Psi^{-1} \circ \Psi')_\lambda = \Psi^{-1}_\lambda \circ \Psi_\lambda'$ is a non-zero morphism $T(\lambda)_\lambda \to T(\lambda)_\lambda$. Since $T(\lambda)_\lambda$ is 1-dimensional, it follows that $z$ is mapped to a non-zero scalar multiple of $z$. This implies that $\Psi^{-1} \circ \Psi'$ cannot be nilpotent.
\end{ex}

\begin{ex}
The argument given in \Autoref{duality_uqmod} also shows that the usual duality $\bbD_{\tau,\vee}$ from \cite[\S3.2]{HumphreysBGG} on the Bernstein--Gelfand--Gelfand category $\mathcal{O}$ of a finite-dimensional complex semisimple Lie algebra is a standard duality. Working directly with the graded duality on (infinite-dimensional) modules avoids having to make use of the fact that category $\mc{O}$ is equivalent to representations of a certain finite-dimensional algebra.
\end{ex}

\begin{ex} \label{hecke_example}
	Let $W$ be a finite complex reflection group \cite{ShephardTodd}. We fix a map $\mathbf{c} \colon \mathcal{S} \to \CC$ from the set $\mathcal{S}$ of reflections in $W$ to the complex numbers which is invariant under $W$-conjugation, and let $H_{1,\mathbf{c}}(W)$ be the rational Cherednik algebra \cite{EG} for $W$ at parameters $\mathbf{c}$ and $t=1$. The algebra $H_{1,\mathbf{c}}(W)$ is an infinite-dimensional $\mathbb{C}$-algebra admitting a triangular decomposition and a corresponding category $\mathcal{O}^{\mathrm{ln}}$ of locally nilpotent modules, as defined in \cite{GGOR}. Let $\mr{eu} \in H_{1,\mathbf{c}}(W)$ be the Euler element as in \cite{GGOR}. Given $M \in \mathcal{O}^{\mathrm{ln}}$ let $M_\alpha$ for $\alpha \in \CC$ be the generalized eigenspace for the action of $\mathrm{eu}$ on $M$ with generalized eigenvalue $\alpha$. Let $\mathcal{O}$ be the full subcategory of $\mathcal{O}^{\mathrm{ln}}$ consisting of modules whose generalized eigenspaces are finite-dimensional and which are equal to the sum of their generalized eigenspaces; this sum is nessecarily direct. By \cite[Lemma~2.22]{GGOR} a module $M$ of $\mathcal{O}^{\mathrm{ln}}$ belongs to $\mathcal{O}$ if and only if it is finitely generated over $H_{1,\mathbf{c}}(W)$.

	 In \cite[Theorem 2.19]{GGOR} it is proven that $\mathcal{O}$ is a highest weight category with standard and costandard objects, and with simple objects parameterized by the set $\Irr W$ of irreducible complex $W$-modules. Hence, $\mathcal{O}$ is a standard category by \Autoref{finite_hwc_has_everything}. By \cite[Theorem 5.15]{GGOR}, there exists a projective object $P \in \mathcal{O}$ and an isomorphism
	\begin{equation}
		\mc{H}_{\mathbf{q}}(W) \simeq \End_{H_{1,\bc}(W)}(P)^{\op}\;,
	\end{equation}
	where $\mc{H}_{\mathbf{q}}(W)$ is the Hecke algebra \cite{Malle-Rouquier-2003} for $W$ at a parameter $\mathbf{q}$ determined from $\mathbf{c}$. This  Hecke algebra is a generalization to complex reflection groups of the Hecke algebra associated to a (finite) Coxeter group, see \cite{GeckPfeiffer}. The object $P$ is a tilting object in $\mc{O}$ by \cite[Proposition 5.21]{GGOR}. We conclude from \Autoref{end_tilting_is_cellular} that $\mc{H}_{\mathbf{q}}(W)$ admits a standard basis coming from ``factorizing endomorphisms of $P \in \mathcal{O}$ through indecomposable tilting objects of $\mathcal{O}$''.	In particular, $\mc{H}_{\mathbf{q}}(W)$ admits associated cell modules as mentioned in \Autoref{rem_on_cell_modules}. The parameter $\mathbf{q}$ is expressed as the ``exponential'' of $\mathbf{c}$ -- in this way we obtain all parameters $\mathbf{q}$ for the Hecke algebra.

	Now, suppose that $W$ is a \emph{real} reflection group; equivalently, a finite Coxeter group. In this case $\mathcal{H}_{\mathbf{q}}(W)$ is the usual Hecke algebra associated to $W$. As explained in \cite[Remark~4.9]{GGOR}, category $\mathcal{O}$ admits a duality $\bbD \coloneqq \mathbb{D}_{\tau,\vee}$ exchanging standard and costandard modules. Here $M^{\vee}$ consists of all vectors in $(M^*)^{\tau}$ that are locally nilpotent for $\mathbb{C}[\h^*]_+$. It is important to note that this is the same as all vectors in $(M^*)^{\tau}$ that vanish on all but finitely many weight spaces in $M$. We claim that $\mathbb{D}$ is a standard duality on category $\mathcal{O}$. To prove this, we will again use the strategy from \Autoref{fixed_point_strategy}. As explained in \cite[2.4.1]{GGOR}, $\mathcal{O}$ admits a $\mathbb{C}$-graded lift $\widetilde{\mathcal{O}}$ where the grading is given by the generalised eigenspaces under the action of the Euler element $\mr{eu} \in H_{1,c}(W)$. Let $\lambda \in \Irr W$. If $\mathrm{eu}$ acts on $1 \otimes \lambda \subset \Delta(\lambda)$ as multiplication by $c_{\lambda} \in \mathbb{C}$ then condition (\ref{ass_tilting:cond1}) in \Autoref{ass_tilting} implies that the $c_{\lambda}$-weight space of $T(\lambda)$ is equal to $\lambda$ as a $W$-module. Let $\langle -, -  \rangle$, $\langle -,- \rangle'$, $\Psi$ and $\Psi'$ be as in \Autoref{fixed_point_strategy} for $X=T(\lambda)$. Fix a $W$-invariant non-degenerate symmetric bilinear form $( - , - )$ on $\h$. The form induces a $W$-equivariant isomorphism $F \colon \h \to \h^*$. If $y_1, \ds, y_n$ is an orthonormal basis of $\h$ with respect to $( - , - )$, then $\{ x_i := F(y_i) \}_{i}$ and $\{ y_i \}_i$ are dual basis of $\h^*$ and $\h$. As in \cite[\S 4.7]{Baby}, let $\tau$ be the anti-involution of $H_{1,c}(W)$ defined by $\tau(x_i) = y_i, \tau(y_i) = x_i$ and $\tau(w) = w^{-1}$. Applying the formula \cite[page 285]{BEG}, we have
 \begin{equation} \label{tau_euler_equ}
 \tau(\mathrm{eu}) = \tau\left( \frac{1}{2} \sum_{i = 1}^n x_i y_i + y_i x_i \right) = \mathrm{eu}.
 \end{equation}
This implies that $\langle - , - \rangle$ and $\langle - , - \rangle'$ restrict to $W$-associative bilinear forms on each weight space $T(\lambda)_a$, and the weight spaces $T(\lambda)_a$ and $T(\lambda)_b$ are orthogonal for $a \neq b$. In particular, we deduce that $\langle x, - \rangle'$ belongs to $\mathbb{D}(T(\lambda))$ for all $x \in T(\lambda)$. All that remains to be proven is that $\Psi^{-1} \circ \Psi'$ is not nilpotent. We can consider the restrictions $\Psi_{c_{\lambda}}$ and $\Psi_{c_{\lambda}}'$, of $\Psi$ and $\Psi'$ respectively, to the $c_{\lambda}$-weight space. If the associative form $\langle - , - \rangle'$ on $T(\lambda)_{c_{\lambda}} = \lambda$ is degenerate, then it must in fact be zero because $\lambda$ is an irreducible $W$-module. Equation \eqref{eq:SYMofform} shows that this happens precisely when $\langle - , - \rangle$ is a $W$-invariant symplectic form on $\lambda$. Since $W$ is a real reflection group, $\lambda$ also admits a non-degenerate symmetric $W$-invariant bilinear form. Since we are working over the complex numbers, \Autoref{lem:easyinvbilinear} below implies that this is a contradiction. Thus, $\langle - , - \rangle'$ is non-degenerate on $T(\lambda)_{c_{\lambda}}$ and $\Psi_{c_{\lambda}}'$ is an isomorphism. Hence, so too is $(\Psi \circ \Psi')_{c_{\lambda}} = \Psi_{c_{\lambda}} \circ \Psi_{c_{\lambda}}'$. In particular, $\Psi \circ \Psi'$ cannot be nilpotent.
\end{ex}

\begin{lem} \label{lem:easyinvbilinear}
If $W$ is a finite group and $\lambda$ an irreducible complex $W$-module, then there can be (up to scalar) at most one non-zero associative bilinear form on $\lambda$.
\end{lem}

\begin{proof}
Let $\langle - , - \rangle$ and $\langle - , - \rangle'$ be  non-zero associative bilinear forms on $\lambda$ and $\Psi, \Psi'$ the associated morphisms $\lambda \to \lambda^*$. Since $\lambda$ is irreducible, the forms are non-degenerate and $\Psi, \Psi'$ are isomorphisms. By Schur's Lemma, $\Psi^{-1} \circ \Psi'$ is multiplication by a non-zero scalar $c$. Replacing $\langle - , - \rangle$ by $c \langle - , - \rangle$, we may assume $c = 1$. That is, $\Psi = \Psi'$ and hence $\langle - , - \rangle = \langle - , - \rangle'$.
\end{proof}

\begin{rem}\label{rem:cellHeckeGeck}
In this remark we assume that $W$ is a complex reflection group such that the associated generic Hecke algebra carries an anti-involution, as specified in \cite[Assumption~4.2]{GordonGriffChlo}. In this case, it is shown in the proof of \cite[Theorem~4.4]{GordonGriffChlo} that $\Hom_{\mc{O}}(P,\nabla(\lambda))^* \simeq \Hom_{\mc{O}}(P,\Delta(\lambda)) =: S_{\mathbf{q}}(\lambda)$. Therefore, \eqref{eq:standardcostandardiso} implies that the costandard module $\nabla_P(\lambda)$ is isomorphic to the ``standard'' $\mathcal{H}_{\mathbf{q}}(W)$-module $S_{\mathbf{q}}(\lambda)$ of \cite{GordonGriffChlo}.
It follows from \textit{loc. cit.} that $\nabla_P(\lambda)$ carries a symmetric bilinear form which is non-zero if and only if $L_{P}(\lambda) \neq 0$. The results of Du--Rui, as described \Autoref{rem_on_cell_modules}, can be viewed as a way of generalizing \cite[Theorem~4.4]{GordonGriffChlo} to the case where the generic Hecke algebra does not satisfy \cite[Assumption~4.2]{GordonGriffChlo}.
If $W$ is a Coxeter group and we assume Lusztig's conjectures P1--P15 hold, then it follows from \cite[Proposition 4.6]{GordonGriffChlo} that the costandard modules $\nabla_P(\lambda)$ are isomorphic to Geck's cell modules $W_{\mathbf{q}}(\lambda)$.
\end{rem}

\end{document}